\documentclass[11pt]{article}

\usepackage[english]{babel}
\usepackage{amssymb}
\usepackage{amsmath}
\usepackage{hyperref}
\usepackage{enumerate}
\usepackage{amsthm}
\usepackage{amsfonts}
\usepackage{listings}
\usepackage[curve]{xypic}

\usepackage{MnSymbol}
\usepackage{stmaryrd}

\DeclareMathSymbol{\mlq}{\mathord}{operators}{``}
\DeclareMathSymbol{\mrq}{\mathord}{operators}{`'}

\title{Generalized automorphic sheaves and the proportionality principle of Hirzebruch-Mumford}
\date{5.6.2019}
\author{Fritz H\"ormann\\ Mathematisches Institut, Albert-Ludwigs-Universit\"at Freiburg}

\usepackage[numbers,sort&compress]{natbib}

\usepackage{color}
\usepackage[headsepline,footsepline]{scrpage2}

\usepackage{tikz}

\newtheorem{SATZ}{Theorem}[subsection]

\newtheorem{LEMMA}[SATZ]{Lemma}
\newtheorem{DEF}[SATZ]{Definition}
\newtheorem{PROP}[SATZ]{Proposition}

\newtheorem{BEISPIEL}[SATZ]{Example}

\newtheorem{FOLGERUNG}[SATZ]{Corollary}
\newtheorem{KOR}[SATZ]{Corollary}

\newtheorem{BEM}[SATZ]{Remark}

\newtheoremstyle{bare}        % name
  {}            % Space above, empty = `usual value'
  {}            % Space below
  {\normalfont}                 % Body font (\normalfont)
  {}                            % Indent amount (empty = no indent, \parindent = para indent)
  {\bfseries}                   % Thm head font
  {}                            % Punctuation after thm head
  {.0em}                           % Space after thm head: " " = normal interword space;
  {\thmnumber{#2}#1. \thmnote{\normalfont\textsc{(#3)}} } % Thm head spec

\theoremstyle{bare}
\newtheorem{PAR}[SATZ]{}

\newcommand\Tstrut{\rule{0pt}{2.6ex}}         % = `top' strut
\newcommand\Bstrut{\rule[-0.9ex]{0pt}{0pt}}   % = `bottom' strut

\newcommand{\comment}[1]{}

\newcommand{\iso}{\stackrel{\sim}{\longrightarrow}}

\newcommand{\Mat}[1]{ \left(\begin{matrix} #1 \end{matrix} \right) }

\newcommand{\R}{ \mathbb{R} }
\newcommand{\C}{ \mathbb{C} }
\newcommand{\Q}{ \mathbb{Q} }

\newcommand{\Z}{ \mathbb{Z} }

\newcommand{\Gm}{ {\mathbb{G}_m} }
\newcommand{\M}{ {\mathbb{M}} }
\newcommand{\Mm}{ {\mathbb{M}_m} }

\newcommand{\N}{ \mathbb{N} }

\newcommand{\DD}{ \mathbb{D} }

\newcommand{\G}{ \mathbb{G} }
\newcommand{\A}{ \mathbb{A} }
\newcommand{\Af}{ {\mathbb{A}^{(\infty)}} }

\newcommand{\PP}{ \mathbb{P} }
\newcommand{\SSS}{ \mathbb{S} }
\newcommand{\OOO}{\text{\footnotesize$\mathcal{O}$}}
\newcommand{\OO}{ {\cal O} }

\DeclareMathOperator{\Model}{M}
\DeclareMathOperator{\Princ}{B}

\DeclareMathOperator{\Lie}{Lie}
\DeclareMathOperator{\colim}{colim}
\DeclareMathOperator{\Ind}{Ind}

\DeclareMathOperator{\id}{id}
\DeclareMathOperator{\dd}{d}

\DeclareMathOperator{\Hom}{Hom}
\DeclareMathOperator{\Fun}{Fun}

\DeclareMathOperator{\Ext}{Ext}

\DeclareMathOperator{\tr}{tr}
\DeclareMathOperator{\can}{can}
\DeclareMathOperator{\GL}{GL}
\DeclareMathOperator{\spec}{Spec}
\DeclareMathOperator{\spf}{Spf}
\DeclareMathOperator{\pr}{pr}

\DeclareMathOperator{\Res}{Res}

\DeclareMathOperator{\Ad}{Ad}

\newcommand{\cat}[1]{ {[\textnormal{ \textbf{#1} }]} }

\newcommand{\nSD}{\mathbf{X}}

\newcommand{\nSDi}{\mathbf{Y}}
\newcommand{\nSDii}{\mathbf{Z}}

\newcommand{\nP}{P}
\newcommand{\nG}{G}
\newcommand{\nX}{\mathbb{D}}
\newcommand{\nh}{h}

\newcommand{\nU}{U}

\newcommand{\nSh}{\Model}
\newcommand{\nShD}{\Model^\vee}
\newcommand{\nSPB}{\Princ}

\newcommand{\nRPCD}{\Delta}

\begin{document}

\maketitle



\section*{Abstract}

We axiomatize the algebraic properties of toroidal compactifications of (mixed) Shimura varieties and their automorphic vector bundles. A notion of generalized automorphic sheaf is proposed which includes sheaves of (meromorphic) sections of automorphic vector bundles with prescribed vanishing and pole orders along strata in the compactification, and their quotients. These include, for instance, sheaves of Jacobi forms and weakly holomorphic modular forms.
Using this machinery we give a short and purely algebraic proof of the proportionality theorem of Hirzebruch and Mumford.

\tableofcontents

\section{Introduction}

For a (connected) Shimura variety $M$ associated with a reductive group $P$, Hermitian symmetric domain $\DD^+$ and neat arithmetic subgroup $\Gamma \subset P(\Q)^+$, there is a huge supply of (so called) {\bf automorphic vector bundles} on $M$ coming from its 
structure of locally symmetric variety $M = \Gamma \backslash \DD^+$. Each such vector bundle $\Xi^* \mathcal{E}$ is obtained from a $P(\C)$-equivariant vector bundle $\mathcal{E}$
on $M^\vee$, where $M^\vee$ is the compact dual of the Hermitian symmetric domain $\DD^+$. 
The recipe is as follows: There are morphisms of analytic manifolds 
\[ M = \Gamma \backslash \DD^+ \leftarrow \DD^+ \hookrightarrow M^\vee \]
where $M \leftarrow \DD^+$ is the defining $\Gamma$-torsor and the (Borel) embedding $\DD^+ \hookrightarrow M^\vee$ is $P(\R)^+$-equivariant. 
$\Xi^* \mathcal{E}$ is obtained by restricting $\mathcal{E}$ to $\DD^+$ and then taking the quotient by $\Gamma$.

In his seminal work \cite{Hir58} Hirzebruch observed that, if $M$ is compact, the Chern numbers\footnote{All polynomials in the Chern classes of highest degree considered as numbers.} of $\mathcal{E}$ and $\Xi^* \mathcal{E}$ are proportional by a universal rational factor which
may be interpreted as the volume of $M$ w.r.t.\@ a natural volume form. 
Using the theory of toroidal compactifications Mumford \cite{Mum77}  extended this result to non-compact $M$.

The proofs of Hirzebruch and Mumford rely heavily on analytic methods. Since $M$ and $M^\vee$ are both algebraic it is reasonable to expect a purely algebraic proof of the proportionality principle. The theory developed in this article provides such a proof. 
First observe that the construction of automorphic vector bundles is purely algebraic. For consider the right $P(\C)$-torsor (so called standard principal bundle) $M \leftarrow B$ obtained by extension from the $\Gamma$-torsor $M \leftarrow  \DD^+$ (considered as right $\Gamma$-torsor). 
It turns out to be algebraic as well, inducing a diagram
\begin{equation}  \label{diaspb}  \xymatrix{ M & \ar[l]_-\pi B  \ar[r]^-p & M^\vee  } \end{equation}
of algebraic varieties  where $\pi: B \rightarrow M$ is a right-torsor under $P$ and $M^\vee$ is now interpreted as a component of the moduli space of parabolics of $P$ (a flag variety).
The morphism $p$ is $P$-equivariant. The diagram may be seen as a morphism of Artin stacks
\[ \Xi: M \rightarrow \left[ M^\vee / P \right]. \]
If $M$ is non-compact, $M$ has an algebraic toroidal compactification $\overline{M}$ and the morphism $\Xi$ (or equivalently the diagram (\ref{diaspb})) extends
\[ \Xi: \overline{M} \rightarrow \left[ M^\vee / P \right]. \]
The algebraically defined {\bf automorphic vector bundles} are precisely the pull-backs of locally free sheaves on $\left[ M^\vee / P \right]$ (i.e.\@ $P$-equivariant vector bundles on $M^\vee$) along this morphism. 

In this article we axiomatize the situation, extracting a few simple axioms that ultimately imply the proportionality principle of Hirzebruch and Mumford. These axioms are well-known for Shimura varieties, and they have purely algebraic proofs themselves in cases in which $M$ naturally represents a moduli problem of Abelian varieties with extra structure. 

Along the lines, we generalize the notion of automorphic vector bundle in the non-compact case introducing {\bf generalized automorphic sheaves} that include:
\begin{itemize}
\item sheaves of sections of automorphic vector bundles with
certain vanishing conditions along the boundary (e.g.\@ bundles of cusp forms, subcanonical extensions, etc.),
\item the (push-forward of the) structure sheaf $\OO_D$ of the boundary or the structure sheaf $\OO_{\overline{Y}}$ of a closed stratum thereof,
\item line bundles of Jacobi-forms, 
\item the vector bundles $\Omega^i(\overline{M})$ and jet bundles of automorphic vector bundles,
\item line bundles of ``weakly holomorphic'' modular forms (i.e.\@ meromorphic with poles only along at the cusps). 
\end{itemize}
%TODO: Rewrite next paragraph - do not mention Calabi-Yau's etc. 

%This article was motivated by the search for an axiomatization of the algebraic structure of toroidal compactifications of Shimura varieties and their
%automorphic vector bundles, which might also be applied to other moduli spaces (which, for instance, carry certain families of Calabi-Yau threefolds) to study ``modular forms'' on them and their
%``Fourier expansions'' and ``Fourier-Jacobi expansions''  purely algebraically. While this is the content of work in progress, this article focuses on the axiomatization, and explains that the axioms fit the situation
%for toroidal compactifications of Shimura varieties. To show that the language is sufficiently powerful we distill a few simple axioms that imply the famous proportionality theorem of Hirzebruch \cite{Hir58} and Mumford \cite{Mum77}, thus providing a purely algebraic proof thereof. 

We now describe the axiomatization more in detail. All varieties and formal schemes are understood over a field $k$ of characteristic zero. 
We define a {\bf toroidal formal scheme} (Definition~\ref{DEFTOROIDAL}) to be a formal scheme together with an action of $\M_m^n$, where $\M_m$ is the multiplicative monoid on the affine line, which looks like the completion of a (partially) compactified $\G_m^n$-torsor on a variety along a boundary stratum. In other words, they are completions of a sum of line bundles at the zero section with the action of $\M_m^n$ remembered. 
An {\bf abstract toroidal compactification} (Definition~\ref{DEFTC}) is defined as a smooth variety $\overline{M}$ with a divisor of strict normal crossings $D$ together with the structure of toroidal formal scheme on the completions along all strata (of the stratification defined by $D$) in a compatible way w.r.t.\@ the partial ordering of the strata. 
In Section~\ref{SHIMURA1} we explain that toroidal compactifications of mixed Shimura varieties in the sense of Pink \cite{Pink} indeed give rise to such objects.

Moreover, we introduce the notion of {\bf automorphic data} (Definition~\ref{DEFAUTDATA}) on an abstract toroidal compactification. If $D= \emptyset$ this is just the datum of a ``compact dual'' $M^\vee$ and a ``standard principal bundle''
$B$ forming a diagram as (\ref{diaspb}).  

As mentioned above, this situation is well-known in the theory of Shimura varieties. In this case $B$ is called the {\bf standard principal bundle} and is (philosophically) the bundle of trivializations
of the de Rham realization of the universal motive (associated with a representation $\rho$ of the defining group $P$) together with its natural $P$-structure. 
The morphism $p$ in this case is induced by the variation of the Hodge filtration. %The diagram can also be seen as a morphism $\Xi: M \rightarrow \left[ M^\vee / P_M \right]$ to the quotient stack. 
If $M^\vee$ contains a $k$-rational point then the quotient stack is isomorphic to the classifying stack $\left[ \cdot / Q \right]$ of a parabolic $Q \subset P$. Therefore the datum is essentially the same as a $Q$-torsor over $M$. 
%This allows to define a vector bundle $\Xi^* \mathcal{E}$ on $M$ associated with any representation $\mathcal{E}$ of $Q_M$ (or with a $P_M$-equivariant vector bundle on $M^\vee$). %These are called {\bf automorphic vector bundles}. 

This situation generalizes to the case in which $D$ is non-trivial. In this case automorphic data consist of the following: for any stratum $Y$ a diagram
\[ \xymatrix{ C_{\overline{Y}}(\overline{M}) & \ar[l]_-\pi B_Y  \ar[r]^-p & M_Y^\vee    } \]
where $C_{\overline{Y}}$ means formal completion along $\overline{Y}$, and $\pi: B_Y \rightarrow C_{\overline{Y}}(\overline{M})$ is again a right-torsor under a --- now not necessarily reductive --- linear algebraic group $P_Y$ and $M^\vee_Y$ is a component of the moduli space of {\em quasi-}parabolics of $P_M$. The morphism $p$ is again $P_Y$-equivariant. Furthermore the action of $\M_m^{n_Y}$ lifts to $B_Y$ (the lifted action being part of the datum) such that $p$ becomes {\em invariant}. These data have to be functorial w.r.t.\@ the partial ordering of the strata (cf.\@ Definition~\ref{DEFAUTDATA} for the details). 

%For any two strata $Z \le Y$ in addition a closed embedding $P_Z \rightarrow P_Y$ (functorially w.r.t.\@ the partial ordering of strata) which induce an {\em open} embedding $M_Z^\vee \hookrightarrow M_Y^\vee$, and in addition a $P_Z$-equivariant morphism $B_Z \rightarrow B_Y$ compatible with the morphisms to $\overline{M}$ (again functorially w.r.t.\@ the partial ordering of strata). 

Such a datum is present on toroidal compactifications of Shimura varieties. This is probably less well-known, see e.g.\@ \cite{Thesis} and \cite[2.5]{Hor14}.
It exists (philosophically) because the $P_M$-structure of the de Rham realization of the universal motive becomes a $P_Y$-structure near the boundary stratum $\overline{Y}$ (in the formal sense) because of a natural weight filtration on the realization there, leading to a family of mixed Hodge structures.  

The more general situation of an (abstract) toroidal compactification equipped with automorphic data allows one to define {\bf generalized automorphic sheaves}
(Definition~\ref{DEFXI}) on $\overline{M}$. 
For this purpose the category of $P_M$-equivariant vector bundles on $M^\vee$ is not sufficient as input category. Instead, we define a larger {\em Abelian} category, the {\bf Fourier-Jacobi category} (Definition~\ref{DEFFJ}). 
The objects are specified by a
collection of functors
\[ F_Y: \Z^{n_Y} \rightarrow \cat{$\left[ M^\vee_Y / P_Y \right]$-coh} \]
for each stratum $Y$, where $n_Y=\mathrm{codim}(\overline{Y})$ and where $\cat{$\left[ M^\vee_Y / P_Y \right]$-coh}$ denotes the category of (finite dimensional) $P_Y$-equivariant vector bundles on $M^\vee_Y$. These functors are supposed to fulfill a finiteness condition, namely they have to be left Kan extensions of 
functors defined on some bounded subregion of $\Z^{n_Y}$. In particular, the sheaves $F_Y(v+ \lambda e_i)$ become (essentially) constant for sufficiently large $\lambda$ and we require that
they are isomorphic to $F_W(\pr(v))$ restricted to $M^\vee_Y$, where $W$ is a larger stratum.
It is explained in \ref{DEFXI} that such a datum $\{F_Y\}$ defines a coherent sheaf ``$\Xi^*(\{F_Y\})$'' on $\overline{M}$. The essential tool to define those sheaves is the theory of descent on formal/open coverings developed by the author in \cite{Hor16}. This theory enables to glue $\Xi^*(\{F_Y\})$ from sheaves on the various completions. The latter are, by definition, toroidal formal schemes, and the functor $F_Y$ describes the parts of $C_{\overline{Y}}(\Xi^*(\{F_Y\}))$ of varying weight under $\G_m^{n_Y}$.

\vspace{0.3cm}

{\bf Example 1.}
Let $\overline{M}$ be the compactification of a (fine) moduli space of elliptic curves with level structure. There are only two types of strata: $Y=M$ is the open stratum or $Y$ is a point (a cusp). 
In the first case $P_M = \GL_2$ and $M^\vee = \PP^1 = P_M/Q_M$ while in the second case $P_Y = \Mat{{*} & {*}\\ &1}$ and $M^\vee_Y = \A^1 = P_Y/\Gm$. 
The bundle of (weakly holomorphic) modular forms of weight $k$ (with order $\nu_Y \in \Z$ at the cusp $Y$) is given by the following input datum:
\[ F_M := \mathcal{L}^{\otimes k} \]
for the open stratum, where $\mathcal{L}$ is the standard one-dimensional representation of weight $1$ of $Q_M$, and
\[ F_Y: v \mapsto  \begin{cases}  \mathcal{L}^{\otimes k}|_{\A^1} & \text{if }v \ge \nu_Y, \\ 0 & \text{otherwise} \end{cases}
\]
for the cusps. 

\vspace{0.3cm}

{\bf Example 2.}
Let $M'$ be the universal elliptic curve over a  (fine) moduli space of elliptic curves with level structure. Let $\overline{M}$ over $M'$ be the pullback of the Poincar\'e line bundle using the
standard polarization. It is the partial compactification of a $\Gm$-torsor $M$ over $M'$. The variety $M$ is a mixed Shimura variety associated with the group $P_M = \GL_2 \ltimes W$, where $W$ is a Heisenberg group, i.e.\@ a central extension of $\G_a^2$:
\[ \xymatrix{ 0 \ar[r] & U \cong \G_a \ar[r] & W \ar[r] & V \cong \G_a^2 \ar[r] & 0.  } \]
(Here $\GL_2$ acts on $V$ via the natural 2-dimensional representation and on $U$ via the determinant.) In this case there is only one boundary stratum $Y \cong M'$ apart from $M$. Consider the following input datum:
\[ F_M := 0 \]
and
\[ F_Y: v \mapsto  \begin{cases} \mathcal{L}^{\otimes k} & \text{if }v = i, \\ 0 & \text{otherwise.} \end{cases}
\]
for $\mathcal{L}$ as before, extended (as a representation) to the present $Q_M$ in the only possible way. The associated generalized automorphic sheaf is then the bundle of Jacobi forms of weight $k$ and index $i$ (it has support on $Y \cong M'$). Here, for simplicity, we ignored the behaviour along the boundary of $M'$ which can be taken into consideration by using a full compactification of $M$ instead.  

\vspace{0.3cm}

We finally consider the notion of (logarithmic) connection on automorphic data, and certain (purely algebraic) axioms: 
\begin{itemize}
\item[(F)] flatness of the logarithmic connection (\ref{AXIOMSF}),
\item[(T)] infinitesimal Torelli (\ref{AXIOMST}),
\item[(M)] unipotent monodromy condition (\ref{AXIOMSM}),
\item[(B)] boundary vanishing condition (\ref{AXIOMSB}).
\end{itemize}
%(These axioms are of course not all expected to hold in this form for generalizations of the theory to other moduli spaces.)
For example (F) and (T) imply that --- on the open stratum --- the formation of automorphic vector bundles commutes with the formation of sheaves of differential forms and jet bundles (Section~\ref{JET}).
If (M) holds, even the sheaves of differential forms and the jet bundles --- now on the compactification --- can be defined as generalized automorphic sheaves (Section~\ref{JET2}), as opposed to their logarithmic variants which are always usual automorphic vector bundles. 
Finally, if in addition (B) is satisfied, Hirzebruch-Mumford proportionality holds for the compactification (Section~\ref{HMPROP}). In the compact case (M) and (B) are vacuous, and everything becomes much easier.
The reason for the validity of the axioms for automorphic data on toroidal compactifications of (mixed) Shimura varieties is sketched in section \ref{SHIMURA2}.

Finally, we prove the proportionality theorem of Hirzebruch and Mumford in Section~\ref{SECTHM} in the following form: 

\vspace{0.3cm}

{\bf Theorem \ref{THEOREMHMP}.} {\em Let $\overline{M}$ be an abstract toroidal compactification of dimension $n$ equipped with automorphic data with logarithmic connection satisfying the axioms (F, T, M, B) and such that $P_M$ is reductive. 
There is a constant $c \in \Q$ such that for all homogeneous
polynomials $p$ of degree $n$ in the graded polynomial ring $\Q[c_1,c_2, \dots, c_n]$ and all $P_M$-equivariant vector bundles  $\mathcal{E}$ in $\cat{$\left[ M^\vee / P_M \right]$-coh}$
the proportionality
\[ p(c_1(\Xi^*\mathcal{E}), \dots, c_n(\Xi^*\mathcal{E}))  = c \cdot p(c_1(\mathcal{E}), \dots, c_n(\mathcal{E})) \]
holds true. }

\vspace{0.3cm}

The idea of the proof is as follows. Following Atiyah \cite{Atiyah}, the polynomials in the Chern classes of vector bundles can be computed as an element in $H^n(\overline{M}, \omega)\cong k$, resp.\@ 
$H^n(M^\vee, \omega) \cong k$ by a  construction (purely in terms of homological algebra) starting from the extension
\begin{equation}\label{eqintrojet} 
\xymatrix{ 0 \ar[r] &  \Omega^1 \otimes \mathcal{E} \ar[r] & J^1 \mathcal{E} \ar[r] & \mathcal{E} \ar[r] & 0  }
\end{equation}
for $\mathcal{E}$ and for a similar extension for $\Xi^* \mathcal{E}$.
This construction works in every Abelian tensor category. It suffices therefore to find an Abelian tensor category $\mathcal{A}$ which maps via an exact tensor functor to the categories of coherent sheaves
\[ \cat{$\overline{M}$-coh} \text{ and } \cat{$M^\vee$-coh} \] 
respectively, such that an extension like (\ref{eqintrojet}) exists in $\mathcal{A}$ and maps to the extensions $J^1 \mathcal{E}$, and $J^1( \Xi^* \mathcal{E})$, respectively. 
Furthermore, this Abelian tensor category has to satisfy the property that $\Ext^n_{\mathcal{A}}(\OO, \omega')$ is one-dimensional where $\omega'$ is the pre-image of both $\omega_{\overline{M}}$ and $\omega_{M^\vee}$.

In the compact case the category $\cat{$\left[M^\vee / P_M \right]$-coh}$ of $P_M$-equivariant vector bundles on $M^\vee$ can be taken as $\mathcal{A}$.
This does not work in general because $\Xi^* \omega_{M^\vee} = \omega_{\overline{M}}(\log)$ and mostly $H^n(\overline{M}, \omega(\log))=0$. 

In the non-compact case, the Fourier-Jacobi categories can be taken as $\mathcal{A}$. 
Here the boundary vanishing condition comes into play which, by an easy homological algebra argument, implies that $\Ext^n_{\mathcal{A}}(\OO, \omega')$ is indeed one-dimensional. 
(Strictly speaking we only construct the tensor product on a subcategory of ``torsion-free'' objects in the Fourier-Jacobi-categories and show that $\Xi^*$ respects it. For the reasoning above this is however sufficient.)

This article would never have been realized without interesting discussions with Emanuel Scheidegger, whom I would like to thank very much.
Special thanks to Wolfgang Soergel to whom I am indebted for his aid.

\section*{Notation}

We write $[n]$ for the unordered set $\{1, \dots, n\}$ and $\Delta_n$ for the poset $\{1 \le 2 \dots \le n\}$ also
 regarded as a category.  For a scheme, formal scheme, or stack $X$ we write 
 $\cat{$X$-coh}$ (or sometimes $\cat{$\OO_X$-coh}$) for the category of coherent sheaves on $X$ and $\cat{$X$-qcoh}$ for the category of quasi-coherent sheaves.  
%For an algebraic group $G$ we denote by $\mathfrak{g}$ its Lie algebra. 

\section{Toroidal compactifications}\label{SECTIONTC}

\subsection{Toroidal formal schemes}

\begin{PAR}
Let $k$ be a field of characteristic 0, fixed for the whole article. 
Let $\Mm$ be $\A^1$ with its unital multiplicative monoid structure over $k$ and, as usual, let $\Gm \hookrightarrow \Mm$ be the open subscheme of the multiplicative group. Denote by $\varepsilon$ the unit of $\Mm$ or $\Gm$ and by $\mu$ the multiplication. 

Let $n$ be a positive integer and let $X$ be a formal scheme over $k$ with an action of $\M_m^n$, i.e. with a given morphism
\[\xymatrix{ \M_m^n \times X \ar[r]^-\rho & X } \]
such that the diagram
\[ \xymatrix{
\M_m^n \times \M_m^n \times X \ar[rr]^-{\id \times \rho} \ar[d]^{\mu \times \id} && \M_m^n \times X \ar[d]^\rho \\
\M_m^n \times X \ar[rr]^-\rho && X
} \]
is commutative and such that the composition
\[\xymatrix{X \ar[r]^-{\varepsilon \times \id} & \M_m^n \times X \ar[r]^-\rho & X } \]
is the identity. By restriction along $\G_m^n \hookrightarrow \M_m^n$ there is, in particular, also a $\G_m^n$-action on $X$. 
\end{PAR}

We have the following lemma whose proof we leave to the reader. 

\begin{LEMMA}
Let $X=\spf R$ be an affine formal scheme over $k$.
It is equivalent to give an action of $\M_m^n$ on $X$ or a (topological) {\bf $\Z^n_{\ge 0}$-grading} on $R$, i.e.\@ 
collection of $k$-sub-vectorspaces $R_v \subseteq R$ for each $v \in \Z^n_{\ge 0}$ such that   
\begin{enumerate}
\item 
For all $v, w \in \Z^n_{\ge 0}$, we have 
\[ R_v \cdot R_w \subseteq R_{v+w}.  \]
\item Each $x \in R$ has a {\em unique} expression as a converging sum
\[ x = \sum_{v \in \Z^n_{\ge 0}} x_v \]
with $x_v \in R_v$. 
\end{enumerate}
\end{LEMMA}

We denote by $e_1, \dots, e_n$ the standard basis of $\Z^n$. 

\begin{DEF}\label{DEFTOROIDAL}
A formal $k$-scheme $X$ with an action of $\M_m^n$ is called {\bf toroidal} if there is an affine covering by $\spf R$'s such that
the action restricts to $\M_n^m \times \spf R \rightarrow \spf R$ and such that
\begin{enumerate}
\item All $R_v$ have the discrete topology.
\item The induced map
\[ R_0[R_{e_1}, \dots, R_{e_n}] \rightarrow  R \]
has dense image and induces an isomorphism between the completion of $R_0[R_{e_1}, \dots, R_{e_n}]$ at the ideal $(R_{e_1}, \dots, R_{e_n})$ and $R$. 
\item The $R_{e_i}$ (and hence by 2. all $R_{v}$) are locally free $R_0$-modules of rank 1.
\end{enumerate}
\end{DEF}

It follows that, up to restricting  to a finer open cover, we have 
\[ R \cong R_0 \llbracket x_1, \dots, x_n \rrbracket   \]
with its natural topological $\Z^n_{\ge 0}$-grading. The $x_i$ however are only  determined up to $R_0^\times$.

\begin{PAR}
On a toroidal formal scheme $X$ we also have a ring-sheaf $\OO_{X_0}$ which locally gives the $R_0$'s and the $\OO_{X,v}$ which are coherent $\OO_{X_0}$-submodules of $\OO_X$. The topological space $X$ together with $\OO_{X,0}$ is a scheme and it is isomorphic to the categorical quotient (in the category of formal schemes) of $X$ w.r.t.\@ the action of $\M_m^n$. It is denoted by $X_0$. Furthermore there is an obvious section (a closed embedding) $X_0 \hookrightarrow X$.
\end{PAR}

\begin{BEISPIEL}
The standard example starts from a $\G_m^n$-bundle on a variety which gets partially compactified by glueing in the partial compactification $\G_m^n \hookrightarrow \M_m^n$ and then 
completed at the section given by the origin of $\M_m^n$. 
\end{BEISPIEL}

\subsection{Modules and differentials}

In the following we consider the integers $\Z$ as a category via the natural inclusion of posets into categories. In other words, there is a morphism (and a unique one)
$n \rightarrow n'$ if and only if $n\le n'$. 

\begin{PROP}\label{PROPTCOH}
Let $X$ with an action of $\M_m^n$ be a noetherian toroidal formal scheme. 
It is equivalent to give
\begin{enumerate}
\item a coherent sheaf of $\OO_X$-modules $M$ with an extension of the $\G_m^n$-action (not necessarily the $\M_m^n$-action);
\item a collection of coherent sheaves of $\OO_{X_0}$-modules $M_w$ for $w \in \Z^n$ together with an associative system of multiplication morphisms for $v \in \Z^n_{\ge 0}$:
\[ \OO_{X,v} \otimes_{\OO_{X_0}} M_w \rightarrow M_{v+w} \]
giving for $v=0$ just the module-structure, and such that there are $N', N \in \Z$ with the property that for all $w$ such that for all $i$, if $w_i \ge N$ and $v = e_i$ the morphism is an isomorphism and for all $w$ such that some $w_i < N'$  the module $M_w$ is zero; 

\item a functor with values in coherent sheaves of $\OO_{X_0}$-modules
\begin{eqnarray*} 
 M: \Z^n &\rightarrow& \cat{$\OO_{X_0}$-coh}  \\
 v &\mapsto& M(v)
 \end{eqnarray*}
 such that there are $N, N' \in \Z$ with the property that for all $i$ and for all $v$ with $v_i \ge N$ the morphism $M(v \rightarrow v+ e_i)$ is an isomorphism and for all $v$ such that $v_i<N'$ for some $i$ the module $M(v)$ is zero. 
In other words the functor is isomorphic to the left Kan extension of a functor 
$\Delta_{N-N'}^{n} \rightarrow \cat{$\OO_{X_0}$-coh}$ where $\Delta_{N-N'}$ is considered as an interval $[N', N] \subset \Z$.
\end{enumerate}
\end{PROP}
\begin{proof}[Proof (sketch).]
$1 \leftrightarrow 2$: Given a module $M$ the associated $M_v$ is just the $\OO_{X_0}$-submodule of elements transforming with weight $v$ under $\G_m^n$.
Conversely, the module $M$ is given as the {\em product} of the modules $M_v$.

$2 \leftrightarrow 3$: A collection $M_v$ is associated with the functor $v \mapsto M(v) := M_v \otimes \OO_{X,-v}$.
Here for arbitrary $v \in \Z^n$ we set 
\[  \OO_{X,v} :=  \bigotimes_i \OO_{X,e_i}^{\otimes v_i}. \] 

A morphism $v \rightarrow w$ in $\Z^n$ is mapped to the morphism 
\[ M_v \otimes_{\OO_{X,0}} \OO_{X,-v} \rightarrow  M_w \otimes_{\OO_{X,0}} \OO_{X,-w} \]
induced by
\[ \OO_{X,w-v} \otimes_{\OO_{X,0}} M_v \rightarrow M_{w}. \]
The functoriality of the functor $M$ is equivalent to the associativity of the multiplication on the module $M$. 
\end{proof}

\begin{DEF}\label{DEFTCOH}
Let $X$ with an action of $\M_m^n$ be a noetherian toroidal formal scheme. Coherent $\OO_X$-modules with compatible $\G_m^n$-action as in Proposition~\ref{PROPTCOH}
form an Abelian category which we denote by $\cat{$\OO_X$-tcoh}$. 
\end{DEF}

\begin{LEMMA}
Under the correspondence above, we have that the
$M(v)$ are torsion-free $\OO_{X,0}$-modules and the $M(v \rightarrow w)$ are monomorphisms for all $v \le w$, if and only if $M$ is torsion-free. 
\end{LEMMA}
\begin{proof}Left to the reader. \end{proof}

\begin{BEM}
We define the full subcategory $\Fun(\Z^n, \cat{$\OO_{X_0}$-coh})^{f.g.}$ of $\Fun(\Z^n, \cat{$\OO_{X_0}$-coh})$ as those functors $M$ which %are bounded below and 
have the property stated in Proposition~\ref{PROPTCOH}, 3.
Hence we have an equivalence
\[ \cat{$\OO_X$-tcoh} \cong \Fun(\Z^n, \cat{$\OO_{X_0}$-coh})^{f.g.}. \]
\end{BEM}

\begin{PAR}\label{CONSTTERM}
Let $M$ be a coherent sheaf on $X$ with a compatible action of 
$\G_m^{n}$. We have its  associated functor $M: \Z^{n} \rightarrow \cat{$\OO_{X_0}$-coh}$. As said, there is an $N$ such that $M(\sum \alpha_i e_i)$ is (essentially) constant in $\alpha_i$ if $\alpha_i > N$. We denote this sheaf by $\lim_{\alpha \rightarrow \infty} M(v + \alpha e_i)$. Note that also expressions like $\lim_{\alpha_{1} \rightarrow \infty, \dots, \alpha_{j} \rightarrow \infty} M(v + \alpha_1 e_{i_1} + \cdots + \alpha_j e_{i_j})$ do make sense (up to isomorphism). Given an injection $\beta: [j] \hookrightarrow [n]$ we will regard this construction w.r.t to the {\em missing} indices in the image of $\beta$ as a functor 
\[ \lim_{\beta}: \Fun(\Z^{n}, \cat{$\OO_{X_0}$-coh})^{f.g.} \rightarrow \Fun(\Z^{j}, \cat{$\OO_{X_0}$-coh})^{f.g.} . \]
We just write ``$\lim$'' for this construction w.r.t.\@ all indices. 
\end{PAR}

\begin{PAR}\label{TENSOR1}
For coherent, {\em torsion-free} sheaves $M$ and $N$ we can describe the tensor product $M \otimes N$ with its natural $\M_m^n$ action by the functor
\[  (M \otimes N)(v) = \sum_{v_1+v_2=v} M(v_1)\otimes N(v_2)  \]
where the sum is formed in  $(\lim M) \otimes (\lim  N)$.
\end{PAR}

\begin{PAR} \label{CANEXT}
For any injection $\beta: [j] \hookrightarrow [n]$ define a sheaf $\OO_X[\beta^{-1}]$ as the sheafification of the pre-sheaf defined (for small enough $U$) by
\[ U \mapsto \OO_X(U)[x_{k_1}^{-1}, \dots, x_{k_{n-j}}^{-1}] \] 
where $\{k_1, \dots, k_{n-j}\}$ is the complement of $\mathrm{im}(\beta)$ and the $x_i$ are generators of $\OO_{X,e_i}$.
To a coherent (in the sense of modules on ringed spaces) $\OO_X[\beta^{-1}]$-module with $\G_m^n$-action  we may still associate (in the same way as in Proposition \ref{PROPTCOH}) a functor in $\Fun(\Z^{n}, \cat{$\OO_{X_0}$-coh})$. 
This yields a {\em fully-faithful} functor
\[ \cat{$\OO_X[\beta^{-1}]$-tcoh} \rightarrow \Fun(\Z^{n}, \cat{$\OO_{X_0}$-coh}) \]
which has the property that the functors in the image are {\em constant in the direction of the $e_{k_i}$}. 

The corresponding localization for modules is given by the $\lim$-construction of \ref{CONSTTERM}. More precisely, the diagram
\[ \xymatrix{
\cat{$\OO_{X}$-tcoh} \ar[rr] \ar[d]_{\cong} & &\cat{$\OO_{X}[\beta^{-1}]$-tcoh} \ar@{^{(}->}[d] \\
\Fun(\Z^{n}, \cat{$\OO_{X_0}$-coh})^{f.g.} \ar[r]^-{\lim_{\beta}} & \Fun(\Z^{j}, \cat{$\OO_{X_0}$-coh})^{f.g.} \ar[r]^-{p_{\beta}^*} & \Fun(\Z^{n}, \cat{$\OO_{X_0}$-coh})
} \]
is commutative. Here $p_{\beta}^*$ is the pullback induced by the projection $p_{\beta}: \Z^n \rightarrow \Z^{j}$ induced by $\beta$. The sheaf $\OO_X[\beta^{-1}]$ can be completed afterwards w.r.t.\@ any of the ideals generated by $\OO_{X,e_i}$ for $i \in \mathrm{im}(\beta)$. (For $i \not\in  \mathrm{im}(\beta)$ the completion would be zero.)
This process of inverting elements and completion might be repeated. Any sheaf $R$ of $\OO_X$-algebras so obtained (which still carries an action of $\G_m^n$) still yields a {\em fully-faithful} functor
\[ \cat{$R$-tcoh} \rightarrow \Fun(\Z^{n}, \cat{$\OO_{X_0}$-coh}) \]
whose image consists of functors that are constant in the direction of the $e_{\beta(i)}$ for those $i$ such that (locally) a generator $x_i$ has been inverted. 
An inverse functor on the essential image might be quite complicated to describe. Its values are given as a subset of the infinite  product that was considered in Proposition~\ref{PROPTCOH} but
the sequences might be e.g.\@  bounded below in some direction, point-wise w.r.t.\@ another direction. Since we will not need it we will not elaborate on this. 

A $\G_m^n$-equivariant coherent $\OO_X[\emptyset^{-1}]$-module $\widetilde{M}$ (where $\emptyset: [0] \rightarrow [n]$ is the inclusion of the empty set) is equivalent to just an $\OO_{X_0}$-module via $\widetilde{M} \mapsto \widetilde{M}(0)$. 
Each $\OO_{X_0}$-module $M_0$ in turn
 has a {\bf canonical extension} to an $\OO_X$-Module with $\M_m^n$-action, given by means of the functor 
\[ M_0(v) = \begin{cases} M_0 & \text{if $v\in \Z^n_{\ge 0}$,} \\ 0 & \text{otherwise,}\end{cases} \]
or equivalently by $M := M_0 \otimes_{\OO_{X,0}} \OO_X$ with its natural $\M_m^n$-action. We denote the full subcategory of $\cat{$\OO_X$-tcoh}$ 
consisting of canonical extensions by $\cat{$\OO_X$-tcoh-can}$.
%
%We have a morphism `constant term' of functors:
%\[ \mathrm{c.t.}: M \otimes_{\OO_X} \OO_X[\beta^{-1}] \rightarrow \lim_{\beta} M.  \]
\end{PAR}

\begin{PAR}\label{ATIYAH}
There is the following exact sequence (equivariant w.r.t.\@ the action of $\M_m^{n}$) of coherent sheaves on $X$: 
\[ \xymatrix{ 0 \ar[r] & \widehat{\Omega}_{X_0} \otimes_{\OO_{X_0}} \OO_X \ar[r] & \widehat{\Omega}_X  \ar[r] & \sum_{i} \OO_{X,e_i} \otimes_{\OO_{X_0}} \OO_X  \ar[r] & 0 } \]
where $\sum_{i} \OO_{X,e_i} \otimes_{\OO_{X_0}} \OO_X $ is isomorphic to the  bundle $\widehat{\Omega}_{X/X_0}$. The bundle   
$\widehat{\Omega}_X$ is not a canonical extension. There is the larger bundle $\widehat{\Omega}_{X}(\log)$ which is locally generated by $\widehat{\Omega}_{X}$ and by the rational differentials $\frac{\dd x_i}{x_i}$.
The latter are invariant under the action of $\M_m^{n}$. We proceed to describe the associated functors of the $\M_m^{n}$-equivariant vector bundles $\widehat{\Omega}_{X}$ and $\widehat{\Omega}_{X}(\log)$. 

Consider the Atiyah extensions on $X_0$ associated with the line bundles $\OO_{X,e_i}$ 
\[ \xymatrix{ 0 \ar[r] & \widehat{\Omega}_{X_0} \ar[r] & E_i \ar[r]^-{p_i} & \OO_{X_0} \ar[r] & 0  } \]
and their amalgamed sum
\begin{equation} \label{explicitlog}
\xymatrix{ 0 \ar[r] & \widehat{\Omega}_{X_0} \ar[r] & E \ar[r]^-{\bigoplus p_i} & \bigoplus_i \OO_{X_0} \ar[r] & 0  } 
\end{equation}

Then $\widehat{\Omega}_X(\log)$ is just the canonical extension of $E$, i.e.\@ it is given by the functor
\[ \widehat{\Omega}_X(\log)(v) = \begin{cases} E & \text{if }v \ge 0, \\ 0 & \text{otherwise}. \end{cases} \]
\end{PAR}

In local coordinates one checks the following:

\begin{PROP}
The functor  associated with $\widehat{\Omega}_X$ is given by
\[ \widehat{\Omega}_X(v) = \begin{cases} \{ e \in E \ |\ \forall i:\ v_i =0 \Rightarrow p_i(e) = 0  \} & \text{if }v = \sum v_i e_i \ge 0, \\ 0 & \text{otherwise}, \end{cases} \]
as a subfunctor of $\widehat{\Omega}_X(\log)$.
\end{PROP}

\subsection{Abstract toroidal compactifications}

\begin{PAR}\label{TCSTRAT}
Let $M$ be a smooth $k$-variety. 
Consider an open embedding $M \hookrightarrow \overline{M}$ into a smooth $k$-variety (mostly assumed to be proper), such that
$D:=\overline{M} \setminus M$ is a divisor with strict normal crossings. Consider the coarsest 
statification $\overline{M} = \bigcup_{Y \in \mathcal{S}} Y$ into locally closed subsets such that all components of $D$ are closures of a stratum in the finite set $\mathcal{S}$. The variety $M$ itself will be the unique open stratum. Denote by $n_Y$ the codimension of $\overline{Y}$. 
Consider furthermore a toroidal action $\rho_Y$ of $\M_m^{n_Y}$ on the formal completion $M_Y:=C_{\overline{Y}}(\overline{M})$ of $\overline{M}$ along $\overline{Y}$ which hence establishes $\overline{Y}$ as the invariant subscheme $M_{Y,0}$. For a pair of strata $Y, Z$ we write $Z \le Y$ if $Z \subset\overline{Y}$.
\end{PAR}
\begin{DEF}\label{DEFTC}
The embedding $M \hookrightarrow \overline{M}$ together with the collection $\{\rho_Y\}_Y$ is called a ({\bf partial}, if $\overline{M}$ is not proper) {\bf toroidal compactification} if for each pair $Z \le Y$ of strata we have an  {\em injective} map $\beta_{ZY}: [n_Y] \hookrightarrow [n_Z]$ such that the natural morphism of formal schemes
\[ \xymatrix{
M_Z \ar[r] & M_Y
} \]
is equivariant w.r.t. the action of $\M_m^{n_Y}$, where $\M_m^{n_Y}$ acts via $\beta_{ZY}$ and $\rho_Z$ on $M_Z$. 
\end{DEF}

\begin{BEM}
The map $\beta_{ZY}$ is uniquely determined by the condition in the definition and hence for strata $W \le Z \le Y$ we have
$\beta_{WZ} \beta_{ZY} = \beta_{WY}$.
\end{BEM}

We will regard objects on $\overline{M}$ such as coherent sheaves etc.\@ always with a compatible action of the $\G_m^{n_Y}$ (not necessarily  $\M_m^{n_Y}$) on their completion on 
$M_Y$ for all strata $Y$ in a compatible way. 
\begin{DEF}\label{DEFTCOH}
In particular, let $\cat{$\OO_{\overline{M}}$-tcoh}$ be the category of coherent sheaves with compatible $\G_m^{n_Y}$-actions on the various completions. Denote by 
 $\cat{$\OO_{\overline{M}}$-tcoh-can}$ the full subcategory of those sheaves with compatible $\G_m^{n_Y}$-actions whose completions are all canonical extensions (\ref{CANEXT}). 
\end{DEF}

For example $\Omega^i(\overline{M})$, $T(\overline{M})$ and $\OO_{\overline{M}}$ are naturally objects in $\cat{$\OO_{\overline{M}}$-tcoh}$. 
The former two are not canonical extensions, however. 

\begin{PAR}\label{TCRES}
Each closed stratum $\overline{Y}$ is itself a (partial) toroidal compactification. The completion $C_{\overline{Z}}(\overline{Y})$ is the following formal subscheme
of $C_{\overline{Z}}(\overline{M})$. Its affine pieces are given (with the notation from Definition~\ref{DEFTOROIDAL}) by $R_0 \llbracket R_{e_1}, \dots, R_{e_{n_Z}} \rrbracket $
modulo the ideal generated by $R_{e_{\beta(1)}}, \dots, R_{e_{\beta(n_Y)}}$ (where $\beta=\beta_{ZY}$). The formal scheme $C_{\overline{Z}}(\overline{Y})$ carries an action of $\G_m^{n_Z - n_Y}$. Here
the missing indices not in the image of $\beta$ can be numbered in any way. We denote the corresponding injective map by $\beta_{ZY}^\perp: [n_Z-n_Y] \hookrightarrow [n_Z]$. With the restriction $\beta_{WZ}': [n_Z-n_Y] \hookrightarrow [n_W-n_Y]$ of the transition maps $\beta_{WZ}$ for $W \le Z \le Y$ the scheme $\overline{Y}$ becomes a toroidal compactification. The following commutative diagram shows the compatibility of the chosen numberings: 
\[ \xymatrix{
[n_Z-n_Y] \ar@{^{(}->}[r]^{\beta_{WZ}'} \ar@{^{(}->}[d]^{\beta_{ZY}^\perp} & [n_W-n_Y] \ar@{^{(}->}[d]^{\beta_{WY}^\perp} \\
[n_Z] \ar@{^{(}->}[r]^{\beta_{WZ}} & [n_W]
} \]
\end{PAR}

\begin{LEMMA}
Let $E$ be a coherent sheaf on $\overline{M}$ with compatible $\G_m^{n_Y}$-actions on the respective completions $E_Y$ on $M_Y$. Then for any stratum $Z \le Y$ and $v \in \Z^{n_Y}$ we have that
\[ E_Y(v) \]
is the coherent sheaf on $\overline{Y}$ which (w.r.t. to the restricted structure of toroidal compactification of \ref{TCRES}) corresponds to the functor w.r.t. $Z$:
\[ z \mapsto E_Z(\beta_{ZY}(v) + \beta_{ZY}^\perp(z)). \]
\end{LEMMA}
\begin{proof}Left to the reader. 
\end{proof}

\begin{PAR}
For the following we will work on the topological space underlying $\overline{M}$ itself and consider
coherent sheaves $\mathcal{F}$ on $M_Y$ as coherent $C_{\overline{Y}}(\OO_{\overline{M}})$-modules (in the sense of ringed sheaves) on $\overline{M}$. 
Note that we have
\[ (C_{\overline{Y}}(\OO_{\overline{M}}))(U) = \OO_{M_Y}(U \cap \overline{Y}). \]
Note that this is {\em not} quasi-coherent as $\OO_{\overline{M}}$-module, except for the open stratum $M$ itself. 
We write $C_{\overline{Y}}(\OO_{\overline{M}})|_Y$ for the sheaf
\[ U \mapsto \OO_{M_Y}(Y \cap U) \]
and similarly for a sheaf of $\OO_{M_Y}$-modules $\mathcal{F}$ on $M_Y$ we will write $\mathcal{F}|_Y$ for the so defined restriction considered as a sheaf on $\overline{M}$. 
\end{PAR}

\begin{LEMMA}[Glueing lemma] \label{LEMMAGLUE}
Let the following data be given:
\begin{enumerate}
\item For each stratum $Y$ a functor 
\[ F_Y: \Z^{n_Y} \rightarrow \cat{$\overline{Y}$-tcoh-can}  \]
which satisfies the conditions of Proposition~\ref{PROPTCOH}, 3., 
where $\cat{$\overline{Y}$-tcoh-can}$ is the category of toroidal coherent sheaves on $\overline{Y}$ which are canonical extensions (see \ref{DEFTCOH})\footnote{In fact, only the restriction to $Y$ of these sheaves matter. For
technical reasons --- to be able to describe the glueing --- we consider their canonical extensions here.}. 
\item For all $Z \le Y$ an isomorphism of functors
\begin{equation} \label{eqglue}
 \kappa_{ZY}:  \iota^*_{ZY} F_Y  \iso \lim_{\beta_{ZY}} F_Z    
\end{equation} 
which are compatible w.r.t. $Y \le Z \le W$ in the obvious way. 
Here $\iota_{ZY}: \overline{Z} \hookrightarrow \overline{Y}$ is the natural closed embedding. 
\end{enumerate} 
Then there exists a coherent sheaf $E$ on $\overline{M}$ with compatible actions of $\G_m^{n_Y}$ on $C_{\overline{Y}}(E)$ for all $Y$, with isomorphisms of functors
\[ \lambda_Y:  C_{\overline{Y}}(E)(-)|_Y \cong F_Y(-)|_Y   \] 
which for each $Z \le Y$ are compatible with the functors $\kappa_{ZY}$ in the sense 
that for all $v \in Z^{n_Y}$ the diagram
\begin{equation}\label{diaglueing}
\vcenter{
 \xymatrix{
C_{\overline{Y}}(E)|_Y \ar[r]^-{\lambda_Y} \ar[d] & [F_Y]|_Y \ar@{->>}[d]^{\widetilde{\kappa_{ZY}}} \\
C_{\overline{Y}} (C_{\overline{Z}}(E)[\beta^{-1}_{ZY}])|_Z  \ar[r]^-{\lambda_Z} &  (C_{\overline{Y}} ([ p_{\beta_{ZY}}^* \lim_{\beta_{ZY}} F_Z]))|_Z 
} }
\end{equation}
is commutative. Here $[F_Y]$ is the coherent sheaf of $C_{\overline{Y}} (\OO_{\overline{M}})$-modules determined by the functor $F_Y$, and similarly $[ p_{\beta_{ZY}}^* \lim_{\beta_{ZY}} F_Z]$ is the coherent sheaf of $C_{\overline{Z}} (\OO_{\overline{M}})[\beta_{ZY}^{-1}]$-modules determined by the functor $p_{\beta_{ZY}}^* \lim_{\beta_{ZY}} F_Z$.
The morphism $\widetilde{\kappa_{ZY}}$ is the composition 
\[ [F_Y]|_Y \hookrightarrow C_{\overline{Y}} (C_{\overline{Z}} ([F_Y]) [\beta_{ZY}^{-1}]) \iso C_{\overline{Y}} ([p_{\beta_{ZY}}^* \iota_{ZY}^* F_Y]) \iso  C_{\overline{Y}}  [p_{\beta_{ZY}}^* \lim_{\beta_{ZY}} F_Z] \]
where the second isomorphism is induced by the fact that all $F_Y(v)$ are canonical extensions along $\overline{Z}$ (cf.\@ also \ref{CANEXT}). 
In particular $E$ is isomorphic to $F_M$ on the open stratum $M$.
The sheaf $E$ is uniquely determined (up to unique isomorphism) by this property and the isomorphisms $\kappa$. 
\end{LEMMA}
\begin{proof}
We apply \cite[Main theorem 7.6]{Hor16}. 
The sheaves of $\OO_{\overline{M}}$-algebras $R_Y$ of \cite[7.2]{Hor16} are isomorphic to the restriction of the sheaf $C_{\overline{Y}}(\OO_{\overline{M}})$
to any open subset $U \subset \overline{M}$ such that $U \cap \overline{Y} = Y$, the sheaf that we denote by $C_{\overline{Y}}(\OO_{\overline{M}})|_Y$. 
%All ring sheaves are considered on the topological space underlying $\overline{M}$. 

For any pair of strata $Z \le Y$ the sheaf of $\OO_{\overline{M}}$-algebras $R_{Y,Z}$ of \cite[7.2]{Hor16} is, by definition, equal to 
 $C_{\overline{Y}}(R_Z \otimes_{\OO_{\overline{M}}} \OO_U)$ where $U$ is any open subset such that $U \cap \overline{Y} = Y$ and where the tensor product is
 formed in the category of ring sheaves. 
 The sheaf of $\OO_{\overline{M}}$-algebras $C_{\overline{Y}}(R_Z \otimes_{\OO_{\overline{M}}} \OO_U)$ is also isomorphic to a completion of the
 localization $C_{\overline{Z}}(\OO_X)[\beta_{YZ}^{-1}]$ since $\overline{Y} \setminus Y$ is given in formal local coordinates in $C_{\overline{Z}} (\overline{Y})$  
 by the zero locus of $x_{k_1}, \dots, x_{k_{j}}$ where $\{k_1, \dots, k_{j}\}$ is the complement of $\mathrm{im}(\beta)$. 
 
 By the nature of toroidal compactification of $\overline{M}$ we have an action of $\G_m^{n_Y}$ on $R_Y$ and 
 an action of $\G_m^{n_Z}$ on $R_{Y,Z}$ which are compatible (via $\beta_{ZY}$) with the inclusion
 \[ R_Y \hookrightarrow R_{Y,Z}. \]
 The category of $R_Y$-coherent sheaves with $\G_m^{n_Y}$-action is equivalent to the category
 \[ \Fun(\Z^{n_Y}, \cat{$\OO_{Y}$-tcoh})^{f.g.}. \]
 Hence the given collection of functors $\{F_Y\}_Y$ gives such objects by restricting $F_Y$ to $Y$.
 
 From the category of $R_{Y, Z}$-coherent sheaves  with $\G_m^{n_Z}$-action we have still a fully-faithful embedding into the sub-category of
\[ \Fun(\Z^{n_Z}, \cat{$\OO_{Z}$-tcoh}) \]
consisting of the functors which are constant in the directions $e_i$ for  $i \not\in \mathrm{im}(\beta_{ZY})$. 
%For each $Z \le Y$ we get such an object taking $\lim_{\beta_{ZY}} F_Z$.
The glueing datum required by \cite[Lemma~7.5]{Hor16} can therefore be given by diagram (\ref{diaglueing}). 
Hence, \cite[Main theorem 7.6]{Hor16} provides the requested sheaf of $\OO_{\overline{M}}$-modules which is by construction an object in $\cat{$\OO_{\overline{M}}$-tcoh}$.
\end{proof}

\subsection{Toroidal compactifications of (mixed) Shimura varieties}\label{SHIMURA1}

\begin{PAR}
The standard examples of abstract toroidal compactifications in the sense of Definition~\ref{DEFTC} are toroidal compactifications of Shimura varieties \cite{AMRT}. 
Since we are interested only in the situation over a field, we can use the theory of canonical models of toroidal compactifications of mixed Shimura varieties due to Pink \cite[2.1]{Pink}. We will use the language of \cite{Thesis} (cf.\@ also \cite[2.5]{Hor14}) with is concerned with extensions of the theory over the integers (in the case of good reduction of Hodge type mixed Shimura varieties). For the automorphic data referred to in the next section we rely on \cite[2.5]{Hor14} also for the rational case. In that case the ideas for the proofs of the theorems in \cite[2.5.]{Hor14} (which are given in \cite{Thesis}) are essentially due to Harris \cite{Harris85, Harris86, HZ94}. 
\end{PAR}

\begin{PAR}
For each pure (or mixed) rational Shimura datum $\nSD=(\nP_{\nSD}, \nX_{\nSD}, \nh_{\nSD})$ in the sense of \cite[2.2.3]{Hor14}\footnote{where the integrality property has to be ignored.} or \cite[2.1]{Pink},
and for each sufficiently small compact open subgroup $K \subset \nP_{\nSD}(\Af)$ there is an associated
Shimura variety $\nSh({}^K \nSD)$ which is a smooth quasi-projective variety defined over the reflex field $E(\nSD)$.

Furthermore, for each smooth $K$-admissible rational polyhedral cone decomposition $\nRPCD$ for $\nSD$ (cf.\@ \cite[2.2.23]{Hor14}) there is a (partial) toroidal compactification
$\nSh({}^K_\nRPCD \nSD)$ which contains  $\nSh({}^K \nSD)$ as an open subvariety whose complement is a divisor with strict normal crossings, if $K$ is sufficiently small. 
This and the following is a summary of \cite[Main Theorem~2.5.9]{Hor14}. If $\nRPCD$ is chosen (and this is always possible) to be projective and complete then $\nSh({}^K_\nRPCD \nSD)$ is a smooth projective variety defined over the reflex field $E(\nSD)$.  
This situation thus gives rise to a stratification of $\nSh({}^K_\nRPCD \nSD)$ as considered in \ref{TCSTRAT}. Each stratum corresponds furthermore to an orbit of rational polyhedral cones in $\Delta$. For each stratum $Y$ in this stratification there is a mixed Shimura datum
$\nSDi=(\nP_{\nSDi}, \nX_{\nSDi}, \nh_{\nSDi})$ such that $\nP_{\nSDi}$ is a subgroup of $\nP_\nSD$ (if $\nSD$ is pure, this is a certain normal subgroup of the $\Q$-parabolic of $\nP_\nSD$ describing the corresponding boundary component in the Baily-Borel compactification). The boundary component $\nSDi$ is determined only up to conjugation. Furthermore, $\nRPCD$ restricts to a rational polyhedral cone decomposition $\nRPCD_Y$ for $\nSDi$. 
The partial toroidal compactification of the mixed Shimura variety $\nSh({}^{K_Y}_{\nRPCD_Y} \nSDi)$ has a matching stratum $\widetilde{Y}$ and there is an isomorphism of formal schemes (assuming that $K$ is small enough)
\[ C_{\overline{Y}} \nSh({}^{K}_{\nRPCD} \nSD)  \cong  C_{\overline{\widetilde{Y}}} \nSh({}^{K_Y}_{\nRPCD_Y} \nSDi).  \]
Furthermore, the mixed Shimura variety $\nSh({}^{K_Y} \nSDi)$ is a torus torsor over another mixed Shimura variety $\nSh({}^{K_Y'} \nSDi/U)$  where $U$ is a subgroup of $\nU_{\nSDi}$ (a subgroup of the center of the unipotent radical of $\nP_{\nSDi}$ determined by the mixed Shimura datum) and the action of the torus extends to $\nSh({}^{K_Y}_{\Delta_Y} \nSDi)$ (cf.\@ \cite[2.5.8]{Hor14}). 
The acting torus gets canonically identified with $\G_m^{n_Y}$ (up to numbering of the coordinates) by means of the integral basis of the $n_Y$-dimensional rational polyhedral cone describing $Y$. By construction of the toroidal compactification this 
action extends to $\M_m^{n_Y}$ in such a way that $C_{\overline{\widetilde{Y}}} \nSh({}^{K_Y}_{\nRPCD_Y} \nSDi)$ becomes a toroidal formal scheme in the sense of \ref{DEFTOROIDAL}. The functoriality of the theory implies that the actions of the tori match for pairs of strata $Z \le Y$. Thus $\overline{M}:=\nSh({}^K_\nRPCD \nSD)$ is an abstract toroidal compactification in the sense of Definition~\ref{DEFTC}.
\end{PAR}

\section{Automorphic data}

\subsection{Automorphic data on an abstract toroidal compactification}\label{SECTAUTDATA}

Let $\overline{M}$ be an abstract toroidal compactification (Definition~\ref{DEFTC}). 

\begin{DEF}\label{DEFAUTDATA}
{\bf Automorphic data} on the abstract toroidal compactification $\overline{M}$ consist of a collection $\{ P_Y, M^\vee_Y, B_Y, \dots \}_Y$ indexed by the strata $Y$ of $\overline{M}$ with
\begin{enumerate}
\item a linear algebraic group $P_Y$ (not necessarily reductive);
\item an open and closed subscheme $M_Y^\vee$
of the moduli space of quasi-parabolic subschemes of $P_Y$. We will call these spaces {\bf generalized flag varieties}. If $P_Y$ is reductive then they are projective. We consider the {\em right} action of $P_Y$ on $M_Y^\vee$ by conjugation;
\item a diagram of formal schemes
\[ \xymatrix{ M_Y& \ar[l]_-\pi B_Y \ar[r]^-p & M^\vee_Y  } \]
in which $\pi$ is a right $P_Y$-torsor and $p$ is a $P_Y$-equivariant morphism;
\item a lift of the $\M_m^{n_Y}$-action to $B_Y$ in a $P_Y$-equivariant way, and such that $p$ is $\M_m^{n_Y}$-invariant. We assume that $B_Y$ is a canonical extension, i.e. isomorphic to $\Pi^{-1} B_{\overline{Y}}$ for some bundle on $\overline{Y}$ with its induced $\M_m^{n_Y}$-action, where $\Pi: M_Y \rightarrow \overline{Y}$ is the projection;
(If a $k$-rational point of $M^\vee$ exists, corresponding to a quasi-parabolic $Q_Y$, such a datum is equivalent to a $Q_Y$-principal bundle on $\overline{Y}$.)
\end{enumerate}
together with
\begin{enumerate}
\item[5.] for strata $Z \le Y$ closed embeddings of algebraic groups $\alpha_{ZY}: P_Z \hookrightarrow P_Y$ which induce {\em open embeddings} $M_Z^\vee \hookrightarrow M^\vee_Y$, and $P_Z$- and $\M_m^{n_Y}$-equivariant morphisms $\rho_{ZY}: B_Z \rightarrow B_Y$ such that the diagram of formal schemes
\[ \xymatrix{ M_Z \ar[d] & \ar[l]_-\pi B_Z \ar[d] \ar[r]^-p & M^\vee_Z \ar@{^{(}->}[d]  \\
M_Y & \ar[l]_-\pi B_Y \ar[r]^-p & M^\vee_Y } \]
commutes. The morphisms have to be functorial w.r.t.\@ three strata $W \le Z \le Y$.
\end{enumerate}
\end{DEF}
In other words, if $M^\vee$ contains a $k$-rational point $Q_M$, automorphic data is roughly given by a $Q_M$-torsor on $\overline{M}$ such that the structure group restricts to $Q_Y$ on the formal completion along $\overline{Y}$ in an 
$\M_m^{n_Y}$-equivariant way. Here $Q_Y$ is the quasi-parabolic in $M^\vee_Y(k)$ mapping to $Q_M$. 

\begin{PAR}\label{CLASSAUTOVB}
The diagram \ref{DEFAUTDATA}, 3.\@ for $Y=M$ can be equivalently described by a morphism of Artin stacks (omitting the subscripts $Y$)
\[ \Xi: \overline{M} \rightarrow \left[ P \backslash M^\vee \right].  \]
Let $\mathcal{E}$ be a vector bundle on $\left[ P \backslash M^\vee \right]$, i.e.\@ a $P$-equivariant vector bundle on $M^\vee$.
The pull-back $\Xi^*\mathcal{E}$ is called the automorphic vector bundle associated with $\mathcal{E}$. 
It can be explicitly described as follows: Note that there is a equivalence of categories between $P$-equivariant vector bundles on $B$ and vector bundles on $\overline{M}$. The vector bundle
$\Xi^*\mathcal{E}$ is the vector bundle on $\overline{M}$ corresponding to the $P$-equivariant vector bundle $p^*\mathcal{E}$. 
This construction will be generalized in Section~\ref{FJCAT} (cf.\@ \ref{EXCANEXT} for the special case).
\end{PAR}

\begin{PAR}\label{AXIOMSF}
Consider the following sequence of vector bundles on $B_Y$ (which are all $\M_n^{n_Y}$-equivariant and canonical extensions). 
We assume given a logarithmic Ehresmann connection on $B_Y$, i.e.\@ a section $s_Y$ which is $P_Y$-equivariant and $\M_n^{n_Y}$-equivariant: 
\[ \xymatrix{ 0 \ar[r] & \OO_{B_Y} \otimes \Lie(P_Y) =T_{B_Y}^{\pi-\mathrm{vert}} \ar[r] & T_{B_Y}(\log) \ar[r] & \ar@/_10pt/[l]_-{s_Y} \pi^* T_{M_Y}(\log) \ar[r] & 0. } \]
%Here $\mathfrak{p}_Y$ is the Lie algebra of $P_Y$. 
Note that $P_Y$ acts on $\OO_{B_Y}$ by translation and on  $\Lie(P_Y)$ via $\Ad$. 
Since everything is $\M_n^{n_Y}$-equivariant and a canonical extension, this is equivalent to giving a $P_Y$-equivariant section of the sequence 
\begin{equation}\label{connectionony} 
\xymatrix{ 0 \ar[r] & \OO_{\pi^{-1}\overline{Y}} \otimes \Lie(P_Y) %= \OO_{\pi^{-1}\overline{Y}} \otimes T_{B_Y}^{\pi-\mathrm{vert}} 
\ar[r] & \OO_{\pi^{-1}\overline{Y}} \otimes  T_{B_Y}(\log) \ar[r] & \ar@/_10pt/[l]_-{s_Y'} \pi^* (\OO_{\overline{Y}}  \otimes T_{M_Y}(\log)) \ar[r] & 0. } 
\end{equation}

Furthermore these sections are supposed to be compatible w.r.t.\@ the relation $Z \le Y$ on strata. 
Such a datum will be called {\bf automorphic data with logarithmic connection} on the toroidal compactification $\overline{M}$.
\end{PAR}

\begin{PAR}
We define the $P_Y$-sub-vector bundle $T_{B_Y}^{\mathrm{horz}}$ as the image of $s_Y$, and get a $P_Y$-equivariant decomposition:
\[ T_{B_Y}(\log) = T_{B_Y}^{\pi-\mathrm{vert}}  \oplus  T_{B_Y}^{\mathrm{horz}}. \]

The connection is called {\bf flat}, if 
\begin{enumerate}
\item[(F)] $T_{B_Y}^{\mathrm{horz}}$ is closed under the Lie bracket\footnote{Note that the Lie bracket on $T_{B_Y}$ restricts to $T_{B_Y}(\log)$.}. 
\end{enumerate}
We denote  the corresponding projection operators by $P_\pi^{\mathrm{vert}}$ and $P_\pi^{\mathrm{horz}}$. If $s_Y$ is flat, it induces a homomorphism of ring-sheaves
\begin{equation} \label{eqnu}
\nu: \pi^{-1} \mathcal{D}_{M_Y}(\log) \rightarrow \mathcal{D}_{B_Y}(\log).  
\end{equation}
\end{PAR}

\begin{BEM}
Let $Y$ be a stratum of positive codimension and $D_i$ the components of the divisor with $Y \subset D_i$. 
We have a $P_Y$-equivariant commutative diagram with exact rows and columns
\[ \xymatrix{ 
& & 0 \ar[d] & 0 \ar[d] \\ 
& & \bigoplus_i \OO_{\pi^{-1}\overline{Y}} \ar[d]^{1 \mapsto \xi_{i,Y}'} \ar@{=}[r] & \bigoplus_i \OO_{\pi^{-1}\overline{Y}} \ar[d]^{1 \mapsto \xi_{i,Y}}  \\
0 \ar[r] & \Lie(P_Y) \otimes \OO_{\pi^{-1}\overline{Y}}   \ar@{=}[d] \ar[r] &  T_{B_Y}(\log) \otimes \ar[r] \OO_{\pi^{-1}\overline{Y}}  \ar[d] & \ar@/_10pt/[l]_-{s_Y'} \pi^* (\OO_{\overline{Y}} \otimes  T_{M_Y}(\log)) \ar[r]  \ar[d] & 0.   \\
0 \ar[r] & \Lie(P_Y) \otimes \OO_{\pi^{-1}\overline{Y}}  \ar[r]& T_{\pi^{-1}\overline{Y}} \ar[r]  \ar[d]& %\ar@/_10pt/[l]_-{s_Y''} 
\pi^* T_{\overline{Y}} \ar[r] \ar[d] & 0 \\
& & 0 & 0 } \]
where we denote by $\xi_{i,Y}$, resp.\@ $\xi_{i,Y}'$ the restriction of $x_i \frac{\partial}{\partial x_i}$ for $x_i$ a local equation for $D_i$, resp.\@ $\pi^{-1}D_i$, to $\overline{Y}$, resp.\@ $\pi^{-1}\overline{Y}$. 
Those are independent of the choice of the parameter $x_i$.
%The difference between $s_Y$ applied to this vector field and its canonical lift is a $P_Y$-invariant $\Lie(P)$-valued function  
%on $\overline{Y}'$. This is the residue of the connection.
We have
\[ \Res_{D_i}(s_Y) =   \xi_{i,Y}' - s_Y'(\xi_{i,Y}) = P^{\mathrm{vert}}_\pi(\xi_{i,Y}') \]
which is a $P_Y$-invariant $\Lie(P)$-valued function on $\pi^{-1}\overline{Y}$. This may be taken as the definition of the residue. 
%hence is (up to trivialization ) of determined by a $\Lie(P)$-valued function on $Y$.  
If it is trivial, the datum can be given by a flat connection on the restriction of $B_Y$ to $\overline{Y}$. For strata $Z \le Y$, we have
\begin{equation}\label{rescompstrata} \rho_{ZY}^{-1} (\Res_{D_i}(s_Y))  =  \Res_{D_i}(s_Z). \end{equation}
\end{BEM}

\begin{PAR}\label{AXIOMSM}
Note that, by the structure of toroidal compactification,
we have a sequence dual to sequence (\ref{explicitlog})
\[ \xymatrix{ 0 \ar[r] & \bigoplus_{i=1}^{n_Y} \OO_{M_Y} \cdot \mathrm{can}_{i,M_Y} \ar[r] & T_{M_Y}(\log) \ar[r] & \Pi^* T_{\overline{Y}} \ar[r] & 0 } \]
where $\mathrm{can}_{i,M_Y}$ are the fundamental vector fields for the $\G_m^{n_Y}$-action on $M_Y$, and $\Pi$ is the projection to $\overline{Y}$. 
Similarly for $B_Y$. 

Since $\can_{i,B_Y} |_{\pi^{-1}\overline{Y}} = \xi_{i,Y}'$, we have therefore
\[ \Res_{D_i}(s_Y) =  P^{\mathrm{vert}}_\pi(\mathrm{can}_{i,B_Y}) |_{\pi^{-1}\overline{Y}}. \]

The following axiom will be called the {\bf unipotent monodromy condition}:
\begin{itemize}
\item[(M)] For any $i$, we have $P^{\mathrm{vert}}_\pi(\mathrm{can}_{i,B_Y}) \in \Lie(U^{(i)}) \otimes \OO_{B_Y}$, where
 $\Lie(U^{(i)})$ is a Lie subalgebra of $\Lie(P_Y)$ given by a 1-dimensional {\em normal unipotent} subgroup $\G_a \cong U^{(i)} \subset P_Y$.
\end{itemize}

Since everything is $\M_n$-equivariant, we could state the condition equivalently as $\Res_{D_i}(s_Y) \in \mathfrak{u}_Y^{(i)} \otimes \OO_{\pi^{-1}\overline{Y}}$.
\end{PAR}

\begin{BEM}Axioms (F) and (M) are only concerned with the bundles $M_Y \leftarrow B_Y$. For $k=\C$ 
suppose that $\Pi$ and the local equations $x_i$ of weight $e_i$ converge on $\overline{M}(\C)$ in a neighborhood $U \supset \overline{Y}$. Then for each base point $b \in B$ lying over a point in $U$,
the bundle $B$ with flat connection corresponds to a homomorphism $\pi_1(M) \rightarrow P(\C)$ (Monodromy at $b$). Let $M_i$ be the image in $P(\C)$ of a loop around $D_i$.
 We have then%\footnote{Note that the integral curve of $i \mathrm{can}_{i,B_Y}$ is $2\pi$-periodic.}
\[ M_i = \exp(- 2 \pi \sqrt{-1} \cdot P^{\mathrm{vert}}_\pi(\mathrm{can}_{i,B_Y})(b))  \]
(the choice of $\sqrt{-1}$ corresponds to the orientation of the loop.)
The compatibility (\ref{rescompstrata}) shows that $M_i$ lies in the unipotent subgroup $U^{(i)} \triangleleft P_Y(\C) \subset P(\C)$ for any $Y \subset D_i$. This explains the name of the axiom (M). 	 
\end{BEM}

Axiom (M) has the following immediate consequence:
\begin{LEMMA} \label{LEMMAMONODROMYAXIOM}
We have $p(P^{\mathrm{vert}}_\pi(\mathrm{can}_{i,B_Y})) \in p^* T_{M^\vee}^{(i)}$ (or equivalently $p(P^{\mathrm{horz}}_\pi(\mathrm{can}_{i,B_Y}))  \in p^* T_{M^\vee}^{(i)}$), where
$T_{M^\vee}^{(i)}$ is the subbundle of $T_{M^\vee}$ induced by a Lie subalgebra $\mathfrak{u}_Y^{(i)} \subseteq \Lie(P_Y)$ given by a 1-dimensional {\em normal unipotent} subgroup $\G_a \cong U^{(i)} \subset P_Y$.
\end{LEMMA}
Note that because of the normality of $U^{(i)}$ the bundle $T_{M^\vee}^{(i)}$ is $P_Y$-equivariant itself.

\begin{PAR}\label{AXIOMST}
The automorphic data satisfies {\bf Torelli}\footnote{this rather corresponds to  classical {\em infinitesimal} Torelli theorems}, if we have in addition
\begin{itemize}
\item[(T)]  a direct sum decomposition
\[ T_{B_Y}(\log) = T_{B_Y}^{p-\mathrm{vert}}(\log)  \oplus  T_{B_Y}^{\mathrm{horz}} \]
where $T_{B_Y}^{p-\mathrm{vert}}(\log)$ is the intersection of $T_{B_Y}^{p-\mathrm{vert}}$ with $T_{B_Y}(\log)$ in $T_{B_Y}$.
\end{itemize}
%We call $P_p^{\mathrm{vert}}$ and $P_p^{\mathrm{horz}}$ the corresponding projection operators. 

Since the morphism $\pi^{-1}\overline{Y} \rightarrow M_Y^\vee$ is a submersion (because $P$ maps $\pi^{-1}\overline{Y}$ into itself) $T_{B_Y}(\log) \to p^* T_{M^\vee_Y}$ is still surjective, and we have again an exact sequence with section
\[ \xymatrix{ 0 \ar[r] & T_{B_Y}^{p-\mathrm{vert}}(\log) \ar[r] & T_{B_Y}(\log) \ar[r] & \ar@/_10pt/[l]_-{s} p^* T_{M^\vee_Y} \ar[r] & 0. } \]
whose image is $T_{B_Y}^{\mathrm{horz}}$. 
%Reason: (T) shows that  $T_{B_Y}^{p-\mathrm{vert}}(\log)$ is a subvectorbundle of $T_{B_Y}(\log)$ again

Hence Torelli (T) induces an isomorphism
\[ p^* T_{M^\vee} \cong \pi^* T_M(\log) \] 
and in the same way as before, if $s_Y$ is in addition flat, it induces a homomorphism of ring-sheaves
\begin{equation} \label{eqmu}
\mu: p^{-1} \mathcal{D}_{M_Y^\vee} \rightarrow \mathcal{D}_{B_Y}(\log).  
\end{equation} 
\end{PAR}

\begin{PAR}\label{AXIOMSB}
We also consider the following axiom (called the {\bf boundary vanishing condition}):
\begin{itemize}
\item[(B)] For all strata $Y \not= M$ we have:
$H^i(\left[ M^\vee_Y / P_Y\right], \omega_{M_Y^\vee}) = 0$ for $i \ge \dim(Y)$
\end{itemize}
(cf.\@ Section~\ref{SECTSTACKS} for the notation).
Here $\omega_{M_Y^\vee} = \Omega^{n}_{M^\vee_Y} $ is the highest power of the $P_Y$-equivariant sheaf of differential forms on $M^\vee_Y$. 
\end{PAR}

\subsection{Generalized flag varieties and representations of quasi-parabolic subgroups}\label{SECTSTACKS}

\begin{PAR}
For a linear algebraic group $P$ and a quasi-parabolic subgroup $Q$ we have several functors between $Q$-representations, $P$-representations and (equivariant) coherent sheaves on the quasi-projective variety $M^\vee = Q \backslash P$ (generalized flag variety)\footnote{Hence, in contrast to the last section, we explicitly assume for simplicity that  $M^\vee$ has a $k$-rational point with corresponding quasi-parabolic $Q$.}.
These functors are best understood in the language of Artin stacks. We will not use this theory explicitly but mention it as a guiding principle because it so much clarifies the relations. All representations are, of course, understood to be {\em algebraic}.
We have the following diagram of morphisms of Artin stacks where all stacks are quotient stacks (even schemes in the right-most column):
\begin{equation}\label{eqdiastacks} \vcenter{\xymatrix{
\left[  \cdot / Q \right] \ar[r]_-\sim^-{a} & \left[ M^\vee / P \right]   \ar[d]^-b & M^\vee \ar[d]^-d \ar[l]_-c \\
& \left[   \cdot / P \right] \ar@/^2pt/[r]^-f & \spec(k) \ar@/^2pt/[l]^-e
} }
\end{equation}

We denote the categories of (quasi-)coherent sheaves on a stack $X$ by $\cat{$X$-(q)coh}$ or sometimes by $\cat{$\OO_X$-(q)coh}$. For the particular stacks above, we get
\[
\begin{array}{rl}
\cat{$\left[  \cdot / Q \right]$-coh} & \text{category of finite-dimensional algebraic $Q$-representations in $k$-vector spaces;} \\
\cat{$\left[  \cdot / P \right]$-coh} & \text{category of finite-dimensional algebraic $P$-representations in $k$-vector spaces;} \\
\cat{$\left[ M^\vee / P \right]$-coh} & \text{category  of $P$-equivariant finite dimensional  vector bundles on $M^\vee$;} \\
\cat{$M^\vee$-coh} &  \text{category of coherent sheaves on $M^\vee$;} \\
\cat{$\spec(k)$-coh} &  \text{category of finite-dimensional $k$-vector spaces,}  
\end{array}
\]
and similarly for the categories of quasi-coherent sheaves.

The corresponding pull-back and  (derived) push-forward functors between the categories of (quasi\nobreakdash-)coherent sheaves are given as follows.
\begin{itemize}
\item[$a_*$] associates with a $Q$-representation $V$ a locally free $P$-equivariant sheaf on $M^\vee$. The total space can be described as $(V \times P) / Q$ where $Q$ acts on $V$ and $P$. It defines an equivalence of the category of finite-dimensional $Q$-representations and coherent $P$-equivariant sheaves on $M^\vee$.  
\item[$a^*$] is the inverse of $a_*$, evaluation at the choosen base point of $M^\vee$.
\item[$b_*$] global sections on $M^\vee$, remembering the induced $P$-action. The right derived functors give the cohomology on $M^\vee$ equipped with the induced $P$-action.
\item[$b^*$] associates with a $P$-representation $V$ the  coherent sheaf $V \otimes \OO_{M^\vee}$ with the natural $P$-action. 
\item[$c^*$] forgets the $P$-action.
\item[$d_*$] global sections on $M^\vee$. The right derived functors are the cohomology on $M^\vee$.
\item[$d^*$] associates with a vector space $V$ the coherent sheaf $V \otimes \OO_{M^\vee}$.
\item[$e_*$] induction $\Ind_{\{e\}}^P(-)$, associates with a vector space $V$ the $P$-representation $V \otimes \OO(P)$. 
\item[$e^*$] forgets the $P$-action.
\item[$f_*$] associates with a $P$-representation the vector space of $P$-invariants. This functor is exact if $P$ is reductive. Otherwise the right derived functors are the (Hochschild) group cohomology of $P$ with values in the respective representation.
\item[$f^*$] equips a vector space $V$ with the trivial $P$-representation. 
\end{itemize}
The composed functor $a^* b^*$ is the forgetful functor considering a $P$-representation as a $Q$-representation.
Its right adjoint, the composed functor $b_* a_*$, is therefore also called $\Ind_{Q}^P(-)$ but it is not exact in general.

\end{PAR}

For a stack $X$, 
we denote by $H^i(X, \mathcal{E})$ the higher derived functors of $\pi_*$ evaluated at the (quasi\nobreakdash-)coherent sheaf $\mathcal{E}$, where $\pi$ is the structural morphism.
For example $H^i(\left[  \cdot / P \right], \mathcal{E})$ denotes the (Hochschild) cohomology of $P$ with values in the representation $\mathcal{E}$. 

We will use the following Lemma and its obvious consequences when one of the functors is exact without further mentioning. 

\begin{LEMMA}
For all compositions of push-forward functors along morphisms of Artin stacks we have corresponding Grothendieck spectral sequences of composed functors. 
\end{LEMMA}
\begin{proof}
See e.g.\@ \cite[Tag~070A]{stacks-project}. Cf.\@ also \cite{Jantzen} for more elementary statements regarding the stacks appearing in this section.
\end{proof}

\subsection{Jet bundles on generalized flag varieties}\label{JET}

\begin{PAR}
We start with a general discussion of jet bundles and differential operators.
Let $X$ be a smooth $k$-variety and $X^{(n)}$ the $n$-th diagonal, i.e.\@
\[ X^{(n)} \hookrightarrow X \times X   \]
is the subscheme defined by $\mathcal{J}^n$ 
where $\mathcal{J}$ is the ideal sheaf of the diagonal. Let $\mathcal{E}$ be a vector bundle on $X$.

We have the two projections:
\[ \xymatrix{
& X^{(n)} \ar[rd]^{\pr_2} \ar[ld]_{\pr_1}  \\
X & & X
} \]

One defines the {\bf $n$-th jet bundle} $J^n \mathcal{E}$ by
\[ J^n \mathcal{E} = \pr_{1,*} \pr_2^*\mathcal{E} \]
which is always equipped with a surjective map
\[ J^n \mathcal{E} \rightarrow \mathcal{E}, \]
induced by the unit $\mathcal{E} \rightarrow \Delta_* \Delta^*\mathcal{E}$ where $\Delta: X \hookrightarrow X^{(n)}$ is the diagonal.
Since $\OO_{X^{(n)}} = \pr_1^*\OO_X = \pr_2^*\OO_X$ there is also a splitting of this map in the case $\mathcal{E} = \OO_X$:
\[ \OO_X \rightarrow J^n \OO_X . \]
\end{PAR}

\begin{PAR}
For two vector bundles $\mathcal{E}$ and $\mathcal{F}$ the sheaf of differential operators (of degree $\le n$) is defined as
\[ \mathcal{D}^{\le n}(\mathcal{E}, \mathcal{F}) := \mathcal{HOM}_{\OO_X}( J^n \mathcal{E}, \mathcal{F}). \]

The bundle $J^n\mathcal{E}$ has a second $\OO_X$-module structure coming from $\pr_2$, which we denote as an action on the right.
We have
\[ J^n \OO_X \otimes \mathcal{E} \cong J^n \mathcal{E} \]
where the tensor-product is formed w.r.t.\@ this second $\OO_X$-module structure. 
\end{PAR}

\begin{PAR}
There is an inclusion
\[ \mathcal{D}^{\le n}(\mathcal{E}, \mathcal{F})  \hookrightarrow \mathcal{HOM}_k ( \mathcal{E}, \mathcal{F}) \]
into the sheaf of $k$-linear (not $\OO_X$-linear) morphisms of sheaves. For an open subset $U \subset X$, a section $s \in H^0(U, \mathcal{E})$ here is considered to be a morphism
\[  \OO_U \rightarrow \mathcal{E}_U \]
and the composition
\[  \OO_U \rightarrow \pr_{1,*} \pr_1^* \OO_U = \pr_{1,*} \pr_2^*  \OO_U \rightarrow  \pr_{1,*} \pr_2^*  \mathcal{E}_U  = J^n \mathcal{E}_U \]
yields a section in $H^0(U, J^n \mathcal{E})$ and then, via application of an element of $H^0(U, \mathcal{HOM}( J^n \mathcal{E}, \mathcal{F}))$, a section in $H^0(U, \mathcal{F})$.
The second $\OO_X$-module structure on $J^n\mathcal{E}$ here dualizes  to pre-composition with a section of $\OO_X$. 
We write $\mathcal{D}_X^{\le n} := \mathcal{D}^{\le n}(\OO_X, \OO_X)$. The ring sheaf $\mathcal{D}_X := \colim_n \mathcal{D}_X^{\le n}$ is generated by $\OO_X$ and $\mathcal{T}_X$ with the only relations
coming from the Lie bracket of vector fields and differentiation of functions. 

Similarly to the case of jet bundles, we have
\[ \mathcal{D}^{\le n}(\mathcal{E}, \OO) = \mathcal{D}_X^{\le n} \otimes \mathcal{E}^*  \]
 where the tensor product is formed w.r.t.\@ the right-$\OO_X$-module structure. 
\end{PAR}

\begin{PAR}\label{PARGROUP}
In the special case $X = P$, where $P$ is an algebraic group, we have a natural isomorphism (compatible with the filtration by degree):
\[ \mathcal{D}_P = \colim_{ n } \mathcal{D}^{\le n}_P \cong \OO_P \otimes U(\Lie(P))   \]
where $U(\Lie(P))$ is the universal enveloping algebra of the Lie algebra $\Lie(P)$.
Elements of $\Lie(P)$ are considered to be vector fields using the action by left-translation. They are invariant under the action of $P$ on $P$ by 
right-translation. 
The isomorphism is hence $P$-equivariant under right-translation, where $P$ acts on the right hand side only on $\OO_P$. 
It is $P$-equivariant under left-translation if $G$ on the right hand side acts on $\OO_P$ by left-translation and via $\Ad$ on $\Lie(P)$. 
\end{PAR}

\begin{PAR}
The construction in \ref{PARGROUP} is a special case of the following. Let $P$ be an algebraic group and $X= Q \backslash P$, where $Q$ is a quasi-parabolic subgroup of $P$. 
These are the generalized flag varieties, denoted $M^\vee_Y$  in the
last section thus assuming here that they have a $k$-rational point $[Q]$ in the sequel. Denote by $\pi: P \rightarrow Q \backslash P$ the projection. 
\end{PAR}

\begin{PROP}\label{PROPJET}
Let $E$ be a $Q$-representation and 
\[ \mathcal{E} = Q \backslash (P \times E) \]
the corresponding $P$-equivariant vector bundle on $Q \backslash P$. Then we have
\[ \mathcal{D}(\mathcal{E}^*, \OO) \cong Q \backslash (P \times (U(\Lie(P)) \otimes_{U(\Lie(Q))} E))  \]
where $Q$ acts on $U(\Lie(P))$ via $\Ad$ and on $E$ via the given representation. This isomorphism is compatible with the 
filtration by degree. 
\end{PROP}
\begin{proof}
Sections on $U \subset Q \backslash P$ of the bundle $Q \backslash (P \times (U(\Lie(P)) \otimes_{U(\Lie(Q))} E))$ can be considered as
$Q$-invariant sections on $\pi^{-1}U$ of the constant bundle $U(\Lie(P)) \otimes_{U(\Lie(Q))} E$ 
and similarly sections on $U$ in $\mathcal{E}^*$ are $Q$-invariant sections of the constant bundle $E^*$ on $\pi^{-1}U$. 
The action
\[ H^0(\pi^{-1}U, U(\Lie(P)) \otimes_{U(\Lie(Q))} E) \times H^0(\pi^{-1}U, E^*) \rightarrow H^0(\pi^{-1}U, E^*)  \] 
given by
\[ g (X \otimes v) \cdot f \mapsto g (X v(f)), \]
where $X$ acts as differential operator on the function $v(f) \in \OOO_P({\pi^{-1}U})$, is $Q$-invariant and therefore induces a morphism 
 \[ H^0(\pi^{-1}U, U(\Lie(P)) \otimes_{U(\Lie(Q))} E)^Q \rightarrow \mathcal{D}(\mathcal{E}^*, \OO)(U). \]
Using local coordinates one checks that it is an isomorphism. 
\end{proof}

\begin{DEF}
We define
\[ J^n E :=  (( U(\Lie(P)) \otimes_{U(\Lie(Q))} E^* )^{\le n})^*. \]
\end{DEF}

\begin{FOLGERUNG}[to Proposition~\ref{PROPJET}]
The $P_Y$-equivariant sheaf on $M^\vee_Y$ associated with the representation $J^n E$ is $J^n \mathcal{E}$.
\end{FOLGERUNG}

\begin{PAR}There is a logarithmic version of the sheaves of differential operators defined in the last section. Let $X=\overline{M}$ be
a smooth $k$-variety equipped with a divisor with normal crossings. 
We define 
\[ \mathcal{D}^{\le n}(\OO_X, \OO_X)(\log) \subset \mathcal{D}^{\le n}(\OO_X, \OO_X) \]
as the subsheaf of differential operators generated by $\OO_X$ and the vector fields in $\mathcal{T}_X(\log)$, and  define $\mathcal{D}^{\le n}(\mathcal{E}, \mathcal{F})(\log)$ similarly. We set
\[ J^n_{\log} \mathcal{E} := \mathcal{D}^{\le n}(\mathcal{E}, \mathcal{\OO}_X)(\log)^\vee. \]
\end{PAR}

The following theorem was shown in \cite{Harris86} for the case of Shimura varieties.

\begin{SATZ}\label{THEOREMAUTOVECTDLOG}
Let $\overline{M}$ be a toroidal compactification equipped with automorphic data with logarithmic connection satisfying the axioms $(F, T)$. 
Let $V$ be a representation of $Q_M$, and $\mathcal{V}:= \Xi^* \widetilde{V}$ the corresponding automorphic vector bundle on $\overline{M}$ (cf.\@ \ref{CLASSAUTOVB}).
Then the automorphic vector bundle associated with $J^n V$ is precisely $J_{\log}^n \mathcal{V}$. 
\end{SATZ}
\begin{proof}Let $\widetilde{V}$ denote the bundle $Q \backslash (P \times V)$ on $Q \backslash P$. 
It suffices to show, dually, that the automorphic vector bundle associated with the $P$-equivariant vector bundle $\mathcal{D}^{\le n}(\widetilde{V}^*, \OO)$  on $Q \backslash P$ is $\mathcal{D}^{\le n}(\log)(\mathcal{V}^*, \OO)$.

Let $Y$ be a stratum. For the proof it suffices to take $Y=M$, however, we will need the more refined discussion later. 
There are  $P_Y$-equivariant homomorphisms of ring sheaves (which respect the filtrations by degree), cf.\@ (\ref{AXIOMSF}--\ref{AXIOMST}):
\begin{eqnarray*} 
\mu: \pi^{-1} \mathcal{D}_{M_Y}(\log)& \rightarrow& \mathcal{D}_{B_Y}(\log)  \\
\nu: p^{-1} \mathcal{D}_{M^\vee_Y}& \rightarrow &\mathcal{D}_{B_Y}(\log)   
\end{eqnarray*}
given by the {\em flat} connection $s_Y$ (and the Torelli axiom). 
They are compatible with the left- and right-module structures under $\pi^{-1} \OO_{M_Y}$, resp.\@ $p^{-1} \OO_{M^\vee_Y}$.
Furthermore, we have
\[ \OO_{B_Y} \cdot \nu(p^{-1} \mathcal{D}_{M^\vee_Y}^{\le n}) = \mathcal{D}_{B_Y}^{\mathrm{horz}} = \OO_{B_Y} \cdot \mu(\pi^{-1}  \mathcal{D}_{M_Y}^{\le n}(\log)), \]
where $\mathcal{D}_{B_Y}^{\mathrm{horz}}$ is the sub-ring sheaf of $\mathcal{D}_{B_Y}(\log)$ generated by $\OO_{B_Y}$ and $\mathcal{T}_{B_Y}^{\mathrm{horz}}$. 

 The bundle $\mathcal{D}^{\le n}(\widetilde{V}, \OO)$ on $M^\vee_Y$ is isomorphic to 
\[ \mathcal{D}_{M^\vee_Y}^{\le n} \otimes_{\OO_{M^\vee_Y}} \widetilde{V}^* \]
where the tensor product has been formed w.r.t.\@ the $\OO_{M^\vee_Y}$-right-module structure on $\mathcal{D}_{M^\vee_Y}^{\le n}$.

Furthermore, we have a  $P_Y$-equivariant isomorphism:
\[ p^*( \mathcal{D}_{M_Y^\vee}^{\le n} \otimes_{\OO_{M_Y^\vee}} \widetilde{V}) \cong (\OO_{B_Y} \cdot \mu(\pi^{-1} \mathcal{D}_{M_Y}^{\le n}(\log))) \widehat{\otimes}_{\OO_{B_Y}} p^* \widetilde{V}   \]
(Lemma \ref{LEMMATECH} below).
Now, $P_Y$ acts on $\OO_{B_Y} \cdot \mu(\pi^{-1} \mathcal{D}_{M_Y}^{\le n}(\log))$ exclusively on the first factor, i.e.\@ 
\[ (\OO_{B_Y} \cdot \mu(\pi^{-1}\mathcal{D}_{M_Y}^{\le n}(\log)))^{P_Y} \cong \mathcal{D}_{M_Y}^{\le n}(\log) \]
using the identification of $P_Y$-invariant sections of a $P_Y$-bundle on $B_Y$ with the sections of a vector bundle on $M_Y$.
Conclusion:
\[ (p^*( \mathcal{D}_{M_Y^\vee}^{\le n} \otimes_{\OO_{M^\vee_Y}} \widetilde{V}))^{P_Y} \cong \mathcal{D}_{M_Y}^{\le n}(\log) \widehat{\otimes}_{\OO_{M_Y}} (p^* \widetilde{V})^{P_Y}.   \]
\end{proof}

\begin{LEMMA}\label{LEMMATECH}
The subsheaf $\OO_{B_Y} \cdot \nu(p^{-1}\mathcal{D}_{M^\vee_Y}^{\le n})$ of $\mathcal{D}_{B_Y}(\log)$ is also a right-$\OO_{B_Y}$-submodule sheaf, and we have:
\[ p^*( \mathcal{D}_{M^\vee_Y}^{\le n} \otimes_{\OO_{M^\vee_Y}} \widetilde{V}) \cong (\OO_{B_Y} \cdot \nu(p^{-1}\mathcal{D}_{M^\vee_Y}^{\le n})) \widehat{\otimes}_{\OO_{B_Y}} p^* \widetilde{V}   \]
where the tensor product in both cases is formed w.r.t.\@ the right-module structure. 
\end{LEMMA}
\begin{proof}This follows by induction on the degree from the fact that $\nu$ is compatible with the right-$p^{-1} \OO_{M^\vee_Y }$-module structure. 
\end{proof}

\subsection{Fourier-Jacobi categories}\label{FJCAT}

\begin{DEF}\label{DEFFJ}
Let $\overline{M}$ be a toroidal compactification equipped with automorphic data. We define the Fourier-Jacobi category $\cat{$\overline{M}$-FJ}$ of $\overline{M}$. 
The objects are collections of functors
\[ F_Y: \Z^{n_Y} \rightarrow \cat{$\left[ M^\vee_Y /P_Y\right]$-qcoh} \]
for each stratum $Y$, and natural transformations $\mu_{ZY}$ for each pair $Y\le Z$ of strata, satisfying the following conditions:
\begin{enumerate}
\item For each $j$ there is an $N \in \Z$ such that for all $v$ with $v_j \ge N$ the objects
\[ F_Y(v)   \]
do not depend on $v_j$ and for all $v\le v'$ with $v_j, v_j' \ge N$ the morphisms 
\[ F_Y(v \rightarrow v') \]
do not depend on $v_j$ and $v_j'$ and are identities if $v_i=v_i'$ for all $i \not=j$. In other words, the $F_Y$ are isomorphic to a left Kan extension of a functor
$\Z_{\le N}^{n_Y} \rightarrow  \cat{$\left[ M^\vee_Y /P_Y\right]$-qcoh}$\footnote{This would rather only say that the $F_Y$ become constant {\em up to isomorphism}, but there is no harm in requiring that they are {\em actually} constant.}. 

We denote the respective constant value by $\lim_{\lambda \rightarrow \infty} F_Y(v + \lambda e_j)$.
Note that also expressions like
$\lim_{\lambda_1, \lambda_2 \rightarrow \infty} F_Y(v + \lambda_1 e_j + \lambda_2 e_k)$ etc.\@ make sense. 
\item
For all $Z \le Y$ with corresponding map $\beta_{ZY}: [n_Y] \hookrightarrow [n_Z]$ and morphism $\alpha_{ZY}: P_Z \rightarrow P_Y$ there are isomorphisms
\[ \mu_{ZY}(v): \alpha_{ZY}^* F_Y(v) \iso \lim_{\lambda_{k_1}, \dots, \lambda_{k_l} \rightarrow \infty} F_Z(\beta_{ZY}(v) + \lambda_{k_1} e_{k_1}+\cdots+\lambda_{k_l} e_{k_l})  \]
for all $v \in \Z^{n_Y}$. 
Here $\{k_1, \dots, k_l\}$ is the complement of $\mathrm{im}(\beta_{ZY})$. These isomorphisms are supposed to be natural transformations of functors in $v$ and to be functorial w.r.t.\@ three strata $W \le Z \le Y$. 
\end{enumerate}
The morphisms in the category $\cat{$\overline{M}$-FJ}$ are collections of morphisms of functors $\{F_Y \rightarrow F_{Y}'\}_Y$ for all strata which are compatible with the isomorphisms $\mu_{ZY}(v)$. 

In the same way, we define categories $\cat{$\overline{Y}$-FJ}$, where the objects only consist of functors $F_Z$ for $Z \le Y$. 
We also define $\cat{$Y$-FJ}$, whose objects are just functors $F_Y$ satisfying property 1. 
All Fourier-Jacobi categories are Abelian categories.  %TODO: Explain tensor product... (also in the section about local toroidal schemes)
\end{DEF}

\begin{DEF}
We define the following full subcategories of the Fourier-Jacobi categories:
\begin{enumerate}
\item $\cat{$\overline{M}$-FJ-$\ge$}$: We ask in addition that for each stratum $Y$ there is an $N \in \Z$ such that
\[ F_Y(v) = 0  \]
if {\em some} $v_j < N$. Such elements shall be called {\bf bounded below}. It means that $F_Y$ is actually a left Kan extension from a functor
$\Delta_n^{n_Y} \rightarrow \cat{$\left[M^\vee_Y / P_Y \right]$-qcoh}$ for some $n \in \N$, where $\Delta_n$ is considered as an interval $[N, N+n] \subset \Z$. 
\item $\cat{$\overline{M}$-FJ-coh}$: As before but with the additional condition that $F_Y(v)$ is finite dimensional for all $Y$ and $v$. 
Such elements shall be called {\bf coherent}. 
\item $\cat{$\overline{M}$-FJ-$\ge{}N$}$, $\cat{$\overline{M}$-FJ-$\ge{}N$-coh}$: As before but with fixed $N$. 
\item $\cat{$\overline{M}$-FJ-tf}$: All bounded-below objects, such that in addition for all $v \le w$ the morphism $F_Y(v) \rightarrow F_Y(w)$ is 
a monomorphism. Such elements shall be called {\bf torsion-free}. 
\item $\cat{$\overline{M}$-FJ-lf}$: All torsions-free objects, such that for any $Y$ and any diagram in $\Z^{n_Y}$ of the form
\[ \xymatrix{
v \ar[r] \ar[d] & v+e_i \ar[d] \\
v+e_j \ar[r] & v+e_i+e_j
} \]
the corresponding diagram
\[ \xymatrix{
F_Y(v) \ar[r] \ar[d] & F_Y(v+e_i) \ar[d] \\
F_Y(v+e_j) \ar[r] & F_Y(v+e_i+e_j)
} \]
is Cartesian. Such elements shall be called {\bf locally free}. 
\item $\cat{$\overline{M}$-FJ-lf-coh}$: All locally free and coherent objects. 
\end{enumerate} 
\end{DEF}

\begin{PAR}\label{DEFXI}
Obviously the definition of Fourier-Jacobi category mimics the situation for vector bundles on toroidal compactifications and we now proceed to define an exact functor
\[ \Xi^*: \cat{$\overline{M}$-FJ-coh} \rightarrow \cat{$\overline{M}$-tcoh} \]

as follows: 
For each $F_Y(v) \in \cat{$P_Y$-Vect on $M^\vee_Y$}$ we form $p^*(F_Y(v))^{P_Y}|_{\overline{Y}}$ which is a vector bundle on $\overline{Y}$. 
It carries an action of $\M_m^{n_Z-n_Y}$ on 
\[ C_{\overline{Z}}(p_Y^*(F_Y(v))^{P_Y}|_{\overline{Y}}) \cong (p_Z^*(\alpha_{ZY}^*F_Y(v))^{P_Z})|_{\overline{Y}} \]
which is a {\em canonical extension} (cf.\@ \ref{CANEXT}). 

The so defined functors
\[ F_Y': \Z^{n_Y} \rightarrow \cat{$\overline{Y}$-tcoh-can}  \]
(where $\overline{Y}$ is equipped with its structure as restricted toroidal compactification)
together with the maps induced by the $\mu_{ZY}$ satisfy the requirements of Lemma~\ref{LEMMAGLUE}. Hence we get a
coherent sheaf $\Xi^*(\{F_Y\})$ on $\overline{M}$ which carries a $\G_m^{n_Y}$ action on $C_{\overline{Y}}(\Xi^*(\{F_Y\}))$. 

We call the sheaves in the image of $\Xi^*$  {\bf generalized automorphic sheaves}. 
\end{PAR}

\begin{BEISPIEL}\label{EXCANEXT}
The easiest case is 
\[ \Xi^*V := (p_M^* V)^{P_M} \]
where $V$ is a bundle on $\cat{$M$-FJ-coh} = \cat{$\left[M^\vee / P_M\right]$-coh}$. 
It is a vector bundle which is a canonical extension itself and can be described by the collection of functors
\[ F_Y: v \mapsto \begin{cases} \alpha_{YM}^* V & \text{for } v \in \Z^{n_Y}_{\ge 0} \\ 0 & \text{otherwise}. \end{cases} \]
Sheaves of this form are locally free and are called  {\bf  automorphic vector bundles}. 
\end{BEISPIEL}

\begin{BEM}
The Fourier-Jacobi categories are related to the classical Fourier-Jacobi expansions as follows. For each $F \in \cat{$\overline{M}$-FJ}$ and stratum $Y$ there is a morphism {\bf Fourier-Jacobi expansion}:
\[ H^0(\overline{M}, \Xi^* F) \rightarrow \prod_{v \in \Z^{n_Y}} H^0(\overline{M}, \Xi^* F_v), \]
where $F_v$ is the following element of $F \in \cat{$\overline{M}$-FJ}$. On $Y$ it is defined by
\[ F_{v,Y}(w) = \begin{cases} F_Y(v) & \text{ for $w=v$,} \\ 0 & \text{otherwise} \end{cases} \]
and is a similar restriction of $F$ on strata $Z \le Y$ and 0 on all other. Note that $\Xi^* F_v$ has support on $\overline{Y}$. 
\end{BEM}

\begin{DEF}\label{TENSOR2}
For the category $\cat{$\overline{M}$-FJ-tf-coh}$ we define a tensor product mimicing the tensor product of \ref{TENSOR1}.
Let $F$ and $G$ be objects of $\cat{$\overline{M}$-FJ-tf-coh}$. We define
\[ (F \otimes G)_Y: v \mapsto \sum_{v_1+v_2 = v} F_Y(v_{1}) \otimes G_Y(v_{2}) \]
where the sum is formed in $(\lim_{v \rightarrow \infty} F_Y(v)) \otimes (\lim_{v \rightarrow \infty} G_Y(v))$. 
\end{DEF}

\begin{LEMMA} The exact functor (cf.\@ \ref{DEFXI})
\[ \Xi^*: \cat{$\overline{M}$-FJ-coh} \rightarrow \cat{$\overline{M}$-tcoh} \]
preserves the tensor product when restricted to $\cat{$\overline{M}$-FJ-tf-coh}$.
\end{LEMMA}
\begin{proof}It suffices to see this on the open parts $M_Y|_Y$ of the $M_Y$. The verification is left to the reader.\end{proof}

\begin{PAR}\label{EXTFUNCTORS}
For each pair $(Y, v)$ where $Y$ is a stratum and $v \in \Z^{n_Y}$ there exist restriction functors:
\[
\begin{array}{rrcl} 
(v)_Y^*: & $\cat{$\overline{M}$-FJ-$\ge{}N$-coh}$ &\rightarrow&  \cat{$\left[M^\vee_Y/P_Y\right]$-coh}   \\
(v)_Y^*: & $\cat{$\overline{M}$-FJ}$ &\rightarrow&  \cat{$\left[M^\vee_Y/P_Y\right]$-qcoh}   \\
(v)_Y^*: & $\cat{$\overline{M}$-FJ-$\ge{}N$}$ &\rightarrow&   \cat{$\left[M^\vee_Y/P_Y\right]$-qcoh}   
\end{array}
\]
given by $F \mapsto F_Y(v)$.
Those are exact and have each an {\em exact right-adjoint} $(v)_{Y,*}$ which is given as follows. The functor $((v)_{Y,*}V)_Y$ is given by the right Kan-extension $v_*$, where $v: \{\cdot\} \hookrightarrow \Z^{n_Y}$, resp.\@  $v: \{\cdot\} \hookrightarrow \Z^{n_Y}_{\ge N}$ also denotes the inclusion of $v$. In other words, we have
\[ ((v)_{Y,*}V)_Y(w) = \begin{cases}
V & \text{if $w \le v$ (and $w_i \ge N$ for all $i$ in the $\ge N$-cases)} \\
0 & \text{otherwise.}
\end{cases} \]
Note that $v \le w$ means that $v_i \le w_i$ for all $i$. 
For any stratum $Z \le Y$ we define
\[ ((v)_{Y,*}V)_Z(v) := \alpha^*_{ZY} ((v)_{Y,*}V)_Y(\pr(v))  \]
where $\pr: \Z^{n_Z} \rightarrow \Z^{n_Y}$ is the projection induced by $\beta_{ZY}$. In the bounded case it is set identically zero if $v_i < N$ for some $i$. 
For all other strata $Z$ the functor $((v)_{Y,*}V)_Z$ is set identically zero. The so defined object $(v)_{Y,*}V$ together with the obvious isomorphisms satisfies conditions 1.\@ and 2.\@ of the definition of the Fourier-Jacobi category (Definition~\ref{DEFFJ}). 
\end{PAR}

\begin{PAR}
For each stratum $Y$ and each $N \in \Z$, there are exact restriction functors
\[ \iota_N^*: \cat{$\overline{Y}$-FJ-coh} \rightarrow \cat{$\overline{Y}$-FJ-$\ge{}N$-coh}\]
which have an {\em exact left-adjoint}
\[ \iota_{N,!}:  \cat{$\overline{Y}$-FJ-$\ge{}N$-coh} \hookrightarrow \cat{$\overline{Y}$-FJ-coh} \]
which is given by the natural inclusion (or, in other words, by extension by zero or left Kan extension for the individual $F_Z$).
\end{PAR}

\begin{KOR}\label{KOREMBEDDINGD}
For each stratum $Y$, integer $N$, and $v \in \Z^{n_Y}_{\ge N}$, there are {\em fully-faithful} functors of categories
\[ (v)_{Y,*}: D^\medstar(\cat{$\left[M^\vee_Y/P_Y\right]$-coh}) \hookrightarrow D^\medstar( \cat{$\overline{M}$-FJ-$\ge{}N$-coh})  \]
and
\[ \iota_{N,!}: D^\medstar(\cat{$\overline{M}$-FJ-$\ge{}N$-coh}) \hookrightarrow D^\medstar( \cat{$\overline{M}$-FJ-coh})  \]
for $\medstar\in\{b,+,-,\emptyset\}$. 
\end{KOR}
\begin{proof}We have in each case a pair of adjoint functors in which the unit, resp.\@ the counit, is an isomorphism. Since all four functors are exact, they
induce functors on the derived categories without modification, and form again pairs of adjoint functors (because the counit/unit-equations still hold). Since also the unit, resp.\@ the counit, is still an isomorphism we get the requested fully-faithfulness of the left- (resp.\@ right-) adjoint. 
\end{proof}

In particular, for $Y=M$ and $N=0$ we get that the canonical extension functor $\iota_{0,!}\,(0)_{M,*}$ (cf.\@ Example~\ref{EXCANEXT}) is fully-faithful on the level of derived categories.

\begin{BEM}The statement of Corollary~\ref{KOREMBEDDINGD} is also true for the functors
\[ (v)_{Y,*}: D^\medstar(\cat{$\left[M^\vee_Y/P_Y\right]$-qcoh}) \hookrightarrow D^\medstar( \cat{$\overline{M}$-FJ-$\ge{}N$})  \]
and
\[ \iota_{N,!}: D^\medstar(\cat{$\overline{M}$-FJ-$\ge{}N$}) \hookrightarrow D^\medstar( \cat{$\overline{M}$-FJ})  \]
for $\medstar\in\{b,+,-,\emptyset\}$. 
\end{BEM}

We also have the following two lemmas, which however will not be needed in the sequel. 

\begin{LEMMA}
The categories $\cat{$\overline{M}$-FJ-$\ge{}N$}$ and $\cat{$\overline{M}$-FJ}$ do have enough injectives (while $\cat{$\overline{M}$-FJ-$\ge$}$ does not in general). 
\end{LEMMA}

\begin{proof}
For any object $F=\{F_Y\}$ we define an injective resolution by
\[ \prod_{(Y, v),v_i \le N_Y} (v)_{Y,*} I((v)_Y^* F)  \]
where $I((v)_Y^* F)$ is an injective resolution of $(v)_Y^* F$ in the category $\cat{$\left[M^\vee_Y / P_Y\right]$-qcoh}$. Note that right-adjoints of exact functors and $\prod$ preserve injective objects. Here $N_Y$ is some appropriate upper bound for the stratum $Y$. Note that because of the bound, the product exists (as opposed to general products in $\cat{$\overline{M}$-FJ-$\ge{}N$}$ and $\cat{$\overline{M}$-FJ}$).
\end{proof}

\begin{LEMMA}
The functors
\[ D^\medstar(\cat{$\overline{M}$-FJ-$\ge{}N$-coh}) \hookrightarrow D^\medstar(\cat{$\overline{M}$-FJ-$\ge{}N$}) \]
\[ D^\medstar(\cat{$\overline{M}$-FJ-coh}) \hookrightarrow D^\medstar(\cat{$\overline{M}$-FJ-$\ge{}$}) \]
are fully-faithful for $\medstar\in\{b,-\}$. 
\end{LEMMA}
\begin{proof}
Follows from (the dual of) \cite[Theorem 13.2.8]{KS06}.
\end{proof}
These two lemmas imply, in particular, that $D^b(\cat{$\overline{M}$-FJ-$\ge{}N$-coh})$ is locally small and therefore also $D^b(\cat{$\overline{M}$-FJ-coh})$, 
because all of its objects lie in the image of one of the fully-faithful embeddings 
$D^b(\cat{$\overline{M}$-FJ-$\ge{}N$-coh}) \hookrightarrow D^b(\cat{$\overline{M}$-FJ-coh})$.

\subsection{Jet bundles in Fourier-Jacobi categories}\label{JET2}

\begin{PAR}\label{ATIYAH2}
We write as usual $M_Y:= C_{\overline{Y}}(\overline{M})$ and $M_Y|_Y$ for the formal open subscheme on $Y$. 
Recall the definition of the vector bundle $\Omega_{\overline{M}}(\log)$ on a variety with a normal crossings divisor. 
Locally the bundle $C_{\overline{Y}}( \Omega_{\overline{M}}(\log))|_Y$ is the bundle $\widehat{\Omega}_{M_Y|_Y}(\log)$ (defined in \ref{ATIYAH}) on the toroidal
formal scheme $M_Y|_Y$, but not on $M_Y$!
Recall from \ref{ATIYAH} the description of the associated functor of $\widehat{\Omega}_{M_Y|_Y}(\log)$ on $M_Y|_Y$.

By Theorem~\ref{THEOREMAUTOVECTDLOG} the vector bundle $\Omega_{\overline{M}}(\log)$ on $\overline{M}$ can therefore be obtained by glueing and is associated with 
the following element in $\cat{$\overline{M}$-FJ-lf-coh}$:
\[ F_Y: v \mapsto \begin{cases} \Omega_{M_Y^\vee} &  \text{if }v \ge 0, \\ 0 & \text{otherwise.} \end{cases} \]
Note that for $Z \le Y$ the restriction $\alpha_{ZY}^* \Omega_{M_Y^\vee}$ is canonically isomorphic to $\Omega_{M_Z^\vee}$ because $\alpha_{ZY}$ is supposed
to be an open embedding by definition.

If the given automorphic data with flat logarithmic connection satisfies the unipotent monodromy condition (M) (cf.\@ \ref{AXIOMSM}) then
the subbundle $\Omega_{\overline{M}}$ can be described by the following functor
\begin{eqnarray} \label{eqfilt1}
F_Y: v \mapsto \begin{cases} \left\{ \xi \in \Omega_{M_Y^\vee}\ \middle|\  \forall i: v_i = 0 \Rightarrow p_i(\xi) = 0  \right\} & \text{if }v  \ge 0, \\ 0 & \text{otherwise.} \end{cases} 
\end{eqnarray}

Here $p_i$ is given as follows: By the unipotent monodromy axiom there are $P_Y$-equivariant subbundles
$T^{(i)}_{M_Y^\vee} \subset T_{M_Y^\vee}$ given by the Lie algebras $\mathfrak{u}_i$ of 1-dimensional normal unipotent subgroups $U_i \subset  G_Y$. 
%(If a base point $Q_Y \in M^\vee_Y$ exists, these subbundles are  induced by the inclusion $\mathfrak{u}_i \rightarrow \Lie(P_Y) / \Lie(Q_Y)$. Since $\mathfrak{u}_i$ is normal, the associated $P_Y$-equivariant bundles do not depend on the choice of base point.) 
The morphism $p_i$ is then defined as the projection dual to this inclusion. 
By the unipotent monodromy axiom (M) we have $\OO_{B_Y} \cdot \pi^{-1}(\can_{i,M_Y}) \cong p^*(T^{(i)}_{M_Y^\vee})$ under the natural $P_Y$-equivariant isomorphism
\[ \pi^* \mathcal{T}_{M_Y}(\log) \cong p^* \mathcal{T}_{M_Y^\vee}.   \]
It follows therefore from the proof of Theorem~\ref{THEOREMAUTOVECTDLOG} that $\Omega_{\overline{M}}$ is associated with this subfunctor. 
\end{PAR}

\begin{PAR}\label{ATIYAH3}Assume for the rest of the section that there exists a $k$-valued point in $M^\vee$ and let $Q_M$ be the corresponding quasi-parabolic subgroup of $P_M$. 
The discussion in \ref{ATIYAH2} enables us to refine Theorem~\ref{THEOREMAUTOVECTDLOG}. Given a $Q_M$-representation $V$ or equivalently a $P_M$-equivariant vector bundle $\widetilde{V}$ on $M^\vee$ we define the object $(J^n\widetilde{V})'$ in $\cat{$\overline{M}$-FJ-lf-coh}$ by
\[ (J^n\widetilde{V})_Y': v \mapsto J^n(\widetilde{V})^{v} \]
where we define a $\Z^{n_Y}$-indexed filtration on $J^n(\widetilde{V})$ induced by the dual of the following $\Z^{n_Y}$-indexed filtration on
$( U(\Lie(P_Y)) \otimes_{U(\Lie(Q_Y))} V^* )^{\le n}$: It is given by the tensor product of the trivial filtration on $V^*$ and
the filtration on $U(\Lie(P_Y))$ which is the quotient of the induced filtration on $T(\Lie(P_Y))$ (tensor algebra) of the following filtration on $\Lie(P_Y)$:
\[ \Lie(P_Y)(v) = \begin{cases} \Lie(P_Y) & v \ge 0 \\ \mathfrak{u}_i & v_i = -1 \text{ and } v_j \ge 0\ \forall j \not= i \\ 0 & \text{otherwise.} \end{cases} \]
(This is essentially the dual of (\ref{eqfilt1}).)

\end{PAR}

\begin{SATZ}\label{THEOREMAUTOVECTD}
Let $V$ be a representation of $Q_M$, and let $\mathcal{V}:= \Xi^* \widetilde{V}$ be the corresponding automorphic vector bundle on $\overline{M}$.
Then the generalized automorphic sheaf associated with the element $(J^n\widetilde{V})'$ in $\cat{$\overline{M}$-FJ-lf-coh}$ is precisely $J^n \mathcal{V}$. 
\end{SATZ}

\begin{PAR}
Define $\omega_{\overline{M}}(\log):= \Lambda^n(\Omega_{\overline{M}}(\log))$, where $n= \dim(M)$. By Proposition~\ref{THEOREMAUTOVECTDLOG}, this is an automorphic line bundle associated with $\omega_{M^\vee}$ 
 and by the above discussion the subbundle $\omega_{\overline{M}} \subset \omega_{\overline{M}}(\log)$ is a generalized automorphic sheaf on $\overline{M}$ given by $\omega = \{\omega_Y\}$ with
\[ \omega_Y: v \mapsto \begin{cases} \omega_{M_Y^\vee} & \text{if }v_i \ge 1\ \forall i, \\ 0 & \text{otherwise}. \end{cases}   \]
In other words it is given by $\iota_{1,!}\ (0)_{M,*}\ \omega_{M^\vee}$, where $(0)_{M,*}$  is considered as a functor with values in $\cat{$\overline{M}$-FJ-$\ge 1$-coh}$.
Note that $\omega_{M_Y^\vee}$ is associated with the $Q_Y$-representation $\Lambda^n (\Lie(P_Y) / \Lie(Q_Y))^*$.
We also define the following generalized automorphic sheaves $\omega_Y$ associated with the functor in $\cat{$Y$-FJ-coh}$:
\[ (\omega_Y)_Y: v  \mapsto \begin{cases} \omega_{M^\vee_Y} & \text{if } v=0, \\ 0 & \text{otherwise}. \end{cases}   \]
It extends (as canonical extension along smaller strata) to an element $\omega_{\overline{Y}}$ in $\cat{$\overline{Y}$-FJ-coh}$ (cf.\@ \ref{EXTFUNCTORS}).
In other words $\omega_{\overline{Y}}$ is given by $\iota_{0,!}\ (0)_{Y,*}\ \omega_{M^\vee_Y}$, where $(0)_{Y,*}$ is considered as a functor with values in $\cat{$\overline{M}$-FJ-$\ge{}0$-coh}$.
\end{PAR}

\begin{LEMMA}\label{PROPRES}
There is an exact sequence in $\cat{$\overline{M}$-FJ-coh}$
\[ \xymatrix{ 0 \ar[r] & \omega   \ar[r] & \omega_{M^\vee} \ar[r] &  \bigoplus_{Y \text{ codim 1 strata}} \omega_{\overline{Y}} \ar[r] & \bigoplus_{Y \text{ codim 2 strata}} \ar[r] \omega_{\overline{Y}} & \cdots  } \]
where the sums go over certain {\em multi-}sets of strata which we will not specify because we do not need them explicitly. 
\end{LEMMA}
\begin{proof}By induction.
\end{proof}

\subsection{Automorphic data on toroidal compactifications of (mixed) Shimura varieties}\label{SHIMURA2}

\begin{PAR}
The toroidal compactifications of (mixed) Shimura varieties are naturally equipped with automorphic data with logarithmic connection in the sense of Definition~\ref{DEFAUTDATA}.
We sketch the relation with the theory of mixed Shimura varieties and their toroidal compactifications in this section, hinting at the reasons for the axioms to be satisfied. 
The boundary vanishing axiom which will be investigated more in detail.

Firstly we may fix the particular boundary components $\nSDi$ (in the sense of mixed Shimura data) in its conjugacy class such that for $Z \le Y$ we get a boundary map $\nSDii \rightarrow \nSDi$, i.e.\@
a closed embedding $\nP_{\nSDii} \hookrightarrow \nP_{\nSDi}$ together with a compatible open embedding $\nX_{\nSDi} \hookrightarrow \nX_{\nSDii}$. 
By \cite[Main Theorem~2.5.12]{Hor14} for each of these boundary components $\nSDi$ there exists a ``compact'' dual $\nShD(\nSDi)$ (which is only proper for $\nSDi = \nSD$, i.e.\@ $Y=M$, if $\nSD$ is itself pure) defined over the reflex field $E(\nSD)$. 
It is of the form $M^\vee_Y$ as in the definition of automorphic data, i.e. it is\@ a $\nP_\nSDi$-equivariant component in the classifying space of quasi-parabolics for $\nP_\nSDi$. 
%Except possibly in the case $Y = M$ if $M$ is already proper (i.e.\@ where there is nothing to compactify) $M^\vee$ is even defined over $\Q$ and there is a $\Q$-rational point in $M^\vee(\nSDB_Y)$, i.e.\@ a quasi-parabolic $Q_Y \subset \nP_{\nSDB_Y}$ such that $M^\vee_Y =  \nP_{\nSDB_Y} / Q_Y$. 
For the definition of automorphic data, we will consider all varieties and groups as schemes over the reflex field $E(\nSD)$. 
\end{PAR}

\begin{PAR}The following is a summary of \cite[Main Theorem~2.5.14]{Hor14}.
For each stratum $Y$ there is a $\nP_{\nSDi,E(\nSD)}$-principal bundle $\nSPB({}^{K_Y}_{\nRPCD_Y} \nSDi)$ over the mixed Shimura variety $\nSh({}^{K_Y}_{\nRPCD_Y} \nSDi)$ together with an equivariant map to the ``compact'' dual:
\[ \xymatrix{ \nSh({}^{K_Y}_{\nRPCD_Y} \nSDi) & \ar[l]_p \nSPB({}^{K_Y}_{\nRPCD_Y} \nSDi)  \ar[r]^\pi & \nShD(\nSDi)  }\]
Because of the functoriality (the torus action comes from a morphism of mixed Shimura data) the morphism $p$ is $\M_m^{n_Y}$-equivariant and the morphism $\pi$ is $\M_m^{n_Y}$-invariant. 
These data are compatible in the sense that if we have strata $Z \le Y$ then there is a commutative diagram
\[ \xymatrix{ 
C_{\overline{Z}} \nSh({}^{K}_{\nRPCD} \nSD) \ar[d] \ar[r]^\sim &  C_{\overline{\widetilde{Z}}} \nSh({}^{K_Z}_{\nRPCD_Z} \nSDii) \ar[d] & \ar[l]_p \ar[d] C_{p^{-1}\overline{\widetilde{Z}}} \nSPB({}^{K_Z}_{\nRPCD_Z} \nSDii)  \ar[r]  & \nShD(\nSDii)  \ar[d]  \\
C_{\overline{Y}} \nSh({}^{K}_{\nRPCD} \nSD) \ar[r]^\sim &  C_{\overline{\widetilde{Y}}} \nSh({}^{K_Y}_{\nRPCD_Y} \nSDi) & \ar[l]_p C_{p^{-1}\overline{\widetilde{Y}}} \nSPB({}^{K_Y}_{\nRPCD_Y} \nSDi)  \ar[r] & \nShD(\nSDi) 
 } \]
 where the maps are functorial w.r.t.\@ relations $W \le Z \le Y$ of strata. 
 
 The flat logarithmic connection can be defined analytically by means of the flat section $\xi$ on the universal cover given as follows:
 \[ \xymatrix{  
 \nX_{\nSDi} \times \nP_{\nSDi}(\Af) / K_Y  \ar[d] \ar[rd]^{\xi: [\tau, g] \mapsto [\tau, 1, g]}   \\
\nP_{\nSDi}(\Q) \backslash  \nX_{\nSDi} \times \nP_{\nSDi}(\Af) / K_Y  &  \nP_{\nSDi}(\Q) \backslash  \nX_{\nSDi} \times \nP_{\nSDi}(\C) \times \nP_{\nSDi}(\Af) / K_Y \ar[r] \ar[l] & \nP_{\nSDi}(\C) / Q_Y(\C) 
   } \]
It has logarithmic singularities along the extension of $\nSPB({}^{K_Y}_{\nRPCD_Y} \nSDi)$ to $\nSh({}^{K}_{\nRPCD} \nSD)$ and by GAGA is therefore algebraic. 
The fact that the corresponding algebraic connection is defined over $E(\nSD)$ can be deduced from \cite[3.4]{Harris85}. In purely algebraic constructions of Shimura varieties as moduli spaces it comes from the
Gauss-Manin connection on the cohomology bundle and thus can be constructed in a purely algebraic way. 
\end{PAR}
 
\begin{PAR}
The Torelli axiom (T) follows analytically because the composition
 \[ 
 \nX_{\nSDi} \times \nP_{\nSDi}(\Af) / K_Y  \rightarrow \nP_{\nSDi}(\C) / Q_Y(\C) 
  \]
is an {\em open embedding} after projection to the first factor (the Borel embedding). In purely algebraic constructions of Shimura varieties the axiom corresponds to infinitesimal Torelli theorems of the parametrized objects which can be proven purely algebraically. 
\end{PAR}

\begin{PAR}
The unipotent monodromy axiom (M) is satisfied because the cone $\sigma$ describing a boundary component sits per definition in $\nU_{\nSDi,\R}(-1)$ and $\nU_{\nSDi} \cong \G_a^u$ is a normal subgroup of $\nP_{\nSDi}$ (cf.\@ e.g.\@ \cite[2.2]{Hor14} for its definition). By construction the fundamental vector fields $\can_i$ of the action of $\G_m^{n_Y}$ on $\nSh({}^{K_Y}_{\nRPCD_Y} \nSDi)$ lifted to the universal cover correspond to the basis-vectors of $(\nU_{\nSDi} \cap K_Y)(-1)$ spanning $\sigma$. In cases in which the mixed Shimura variety is constructed using a moduli problem of 1-motives as in \cite[2.7]{Hor14}, the unipotent monodromy axiom can be read off from the construction.  
\end{PAR}

\begin{PROP}[Boundary vanishing condition (B)]\label{PROPBOUNDARYVANISHING}
Let $\nSDi$ be a mixed Shimura datum (e.g.\@ one of the boundary components $\nSDi$), let $n$ be the dimension of $\nShD(\nSDi)$, 
let $Q$ be one of the quasi-parabolics parametrized by $\nShD(\nSDi)$,
let $\omega$ be the $Q$-representation corresponding to the $\nP_{\nSDi}$-equivariant bundle $\omega_{\nShD(\nSDi)} := \Omega^n_{\nShD(\nSDi)}$ on $\nShD(\nSDi)$, and
let $u$ be the dimension of $\nU_{\nSDi}$. Then we have:
\[ H^i(\left[  \cdot / Q \right], \omega) = 0  \]
for all $i \ge n-u$ provided that $u+v \not=0$. 
\end{PROP}
Note that all boundary strata $Y$ which come from rational polyhedral cones in the unipotent cone of $\nSDi$ satisfy $\dim(Y) \ge n-u$. 
\begin{proof}
W.l.o.g.\@ we may assume that the base field of the category of $Q$-representations is $\C$ and that all algebraic groups involved are defined over $\C$. 
We have the following zoo of connected linear algebraic groups (cf.\@ \cite[2.2]{Hor14} or \cite{Pink}):
\begin{equation*}
\begin{array}{rll}
\SSS & = & \G_m^2, \text{ the Deligne torus} \\
P = \nP_\nSDi& = & G \cdot V \cdot U, \text{ where} \\
G = \nG_\nSDi && \text{is a maximal reductive subgroup}\\
V = V_\nSDi & \cong & \G_a^{2v} \\
U = U_\nSDi & \cong &\G_a^{u} \\
h: \SSS \rightarrow G & & \text{any homomorphism in $\nh_\nSDi(\nX_\nSDi)$, which w.l.o.g.\@ can be assumed to factor via $G$} \Bstrut \\
\hline
R & = & K  \cdot  R^+ = G \cap Q \Tstrut \\
 & & \text{is the parabolic in $G$ (with its Levi decomposition) associated with $h$} \\
R^+, R^- & \cong &\G_a^{n_0} \Bstrut \\ 
\hline 
V & =& V^+ \cdot V^- \Tstrut \\
V^+ &=& Q \cap V  \\
Q &= &R \cdot V^+ \text{ is the quasi-parabolic in $P$ associated with $h$ and defining $M^\vee(\nSDi)$ } 
\end{array}
\end{equation*}

By definition of a mixed Shimura datum the Lie algebras of these groups have the following weights under $\SSS$ (acting via $\Ad \circ h$):

\begin{equation*}
\begin{array}{r|l}
\Lie(U) & (-1,-1)  \\
\end{array}\qquad
\begin{array}{r|l}
\Lie(V^+) & (-1,0)  \\
\Lie(V^-) & (0,-1) \\
\end{array}\qquad
\begin{array}{r|l}
\Lie(R^+) & (-1,1) \\
\Lie(K) & (0,0)  \\
\Lie(R^-) & (1,-1) 
\end{array}
\end{equation*}

We have the following sequence of affine morphisms
\[ M^\vee(\nSDi) =  P / (R \cdot V^+) \rightarrow G \cdot V / (R \cdot V^+) \rightarrow G/R  \]
of relative dimensions $u = \dim(U)$, and $v = \dim(V^-)$, respectively. $G/R$ is a projective flag variety of dimension $n_0 = \dim(R^-)$. 
Note that $\omega$ is isomorphic to the representation (with $Q$ acting via $\Ad$ on the Lie algebras)
\begin{equation}\label{omega} \left( \Lambda^{n_0}  \Lie(R^-) \otimes \Lambda^{v} \Lie(V^-) \otimes \Lambda^{u} \Lie(U) \right)^* .  \end{equation}

STEP 1: We have
\[ H^i(\left[  \cdot / Q \right], \omega) = H^i(\left[  \cdot / (V^+ \cdot R^+) \right], \omega)^K   \]
because $K$ is reductive. Furthermore since $\omega$ is 1-dimensional and hence trivial as a $V^+$ and $R^+$ representation, we have as $K$-representations
\[ H^i(\left[  \cdot / (V^+ \cdot R^+) \right], \omega) = H^i(\left[  \cdot / (V^+ \cdot R^+) \right], \C) \otimes \omega. \]

STEP 2: The subgroups $V^+$ and $R^+$ commute (because there is no part of the Lie algebra of weight $(-2,1)$). Hence $H^i(\left[  \cdot / (V^+ \cdot R^+) \right], \C)$ is just the cohomology of
$\G_a^{n_0+v}$ w.r.t.\@ the trivial representation. 
Hence $H^i(\left[  \cdot / V^+ \cdot R^+ \right], \C) = \Lambda^i (\Lie(V^+)^* \oplus \Lie(R^+)^*)$ as natural $\mathrm{Aut}(V^+ \cdot R^+)$-modules \cite[p.64, Remark 2]{Jantzen}.
Therefore $H^i(\left[  \cdot / (V^+ \cdot R^+) \right], \C) = 0$ for $i > n_0 + v$ and  
\[ H^{n_0+v}(\left[  \cdot / (V^+ \cdot R^+) \right], \C) = \Lambda^{n_0+v} (\Lie(V^+)^* \oplus \Lie(R^+)^*) \cong \C.  \]

STEP 3: Since the last isomorphism is compatible w.r.t.\@ the natural $\mathrm{Aut}(V^+ \cdot R^+)$-actions, we see that $H^{n_0+v}(\left[  \cdot / (V^+ \cdot R^+) \right], \C)$ is one-dimensional of weight 
\[ (v+n_0, -n_0) \]
under $\SSS$. The representation $\omega$ is isomorphic to (\ref{omega})  and hence one-dimensional of weight
\[ ( u-n_0,u+v+n_0). \]
Therefore 
\[ H^{n_0+v}(\left[  \cdot / (V^+ \cdot R^+) \right], \C) \otimes \omega \quad \text{ has weight } \quad (u+v,u+v)  \]
and thus cannot have any $K$-invariants as long as $u+v\not=0$. 
\end{proof}

\section{Hirzebruch-Mumford proportionality}\label{SECTHM}

\subsection{Chern classes}

\begin{PAR}
Let $X$ be a smooth projective complex  variety of dimension $n$. 
There are several ways of constructing the Chern classes of vector bundles on $X$. We will use the following, cf.\@ \cite{Atiyah}. Let $\mathcal{E}$ be a vector bundle on $X$. 
It defines an Atiyah extension (where $J^1$ is the first jet bundle (cf.\@ Section~\ref{JET}))
\[ \xymatrix{ 0 \ar[r] & \Omega^{1}_X \otimes \mathcal{E} \ar[r] & J^1 \mathcal{E} \ar[r] & \mathcal{E} \ar[r] & 0.  } \]
Tensoring with $\mathcal{E}^*$ and pulling back along the unit $\OO_X \rightarrow \mathcal{E}^* \otimes \mathcal{E}$ we get an extension
\[ \xymatrix{ 0 \ar[r] & \Omega^{1}_X \otimes \mathrm{End}(\mathcal{E})  \ar[r] & \mathcal{A} \ar[r] & \OO_X \ar[r] & 0 . } \]

This induces a morphism
\[  \OO_X \rightarrow \Omega^{1}_X \otimes \mathrm{End}(\mathcal{E})[1]   \]
in $D^b(\cat{$\OO_X$-coh})$. The coefficients of the characteristic polynomial of this ``endomorphism'' give morphisms
\[ c_i(\mathcal{E}): \OO_X \rightarrow \Omega^{i}_X[i]. \]
Furthermore, any polynomial $p$ in the graded polynomial ring $\Q[c_1,c_2, \dots, c_n]$ (where $\deg(c_i)=i$) of degree $n$ gives a morphism
\[ p(c_1(\mathcal{E}), \dots, c_n(\mathcal{E})): \OO_X \rightarrow  \Omega^{n}_X[n] =: \omega_X[n]. \]
The corresponding extension $p(c_1(\mathcal{E}), \dots, c_n(\mathcal{E})) \in \Ext^n(\OO_X, \omega_X)$ can be constructed explicitly using only locally free sheaves. 
Using the trace map $\tr: \Ext^n(\OO_X, \omega_X) \rightarrow k$ of Serre duality, we get elements $\tr(p(c_1(\mathcal{E}), \dots, c_n(\mathcal{E}))) \in k$. The compatibility with other
constructions of Chern classes using algebraic cycles shows that even $\tr(p(c_1(\mathcal{E}), \dots, c_n(\mathcal{E}))) \in \Q$. 
\end{PAR}

\subsection{Proportionality}\label{HMPROP}

\begin{SATZ}[Hirzebruch-Mumford proportionality]\label{THEOREMHMP}
Let $\overline{M}$ be an abstract toroidal compactification of dimension $n$ equipped with automorphic data with logarithmic connection satisfying the axioms (F, T, M, B) (cf.\@ Section~\ref{SECTAUTDATA}) and such that $P=P_M$ is reductive. 
There is a constant $c \in \Q$ such that for all homogeneous
polynomials $p$ of degree $n$ in the graded polynomial ring $\Q[c_1,c_2, \dots, c_n]$ and all $P$-equivariant vector bundles  $\mathcal{E}$ in $\cat{$\left[ M^\vee / P \right]$-coh}$
the proportionality
\[ p(c_1(\Xi^*\mathcal{E}), \dots, c_n(\Xi^*\mathcal{E})) = c \cdot p(c_1(\mathcal{E}), \dots, c_n(\mathcal{E}))   \]
holds true. 
\end{SATZ}

\begin{proof}
Starting from the sequence in $\cat{$\overline{M}$-FJ-coh}$ (cf.\@ \ref{ATIYAH3} for the definition of $J^1( \mathcal{E})' $):
\[ \xymatrix{ 0 \ar[r] & (\Omega^{1})' \otimes \mathcal{E} \ar[r] & J^1( \mathcal{E})' \ar[r] & \mathcal{E} \ar[r] & 0  } \]
by the procedure described in the last section we can construct an element
\[ \widetilde{p}(\mathcal{E}) \in \Ext^n_{\cat{$\overline{M}$-FJ-coh}}(\OO, \omega).   \]
Note that in the construction only the tensor product of locally free objects is involved and the exactness of $\otimes$ on sequences involving those. 

Consider the following two compositions of functors
\[ \xymatrix{
D^b(\cat{$\overline{M}$-FJ-coh})  \ar[rrr]^{\mathcal{D}^b(\Xi^*)}  &&& D^b(\cat{$\OO_{\overline{M}}$-coh}) \\
D^b(\cat{$\overline{M}$-FJ-coh}) \ar[r]^{(0)_M^*} & D^b(\cat{$M$-FJ-coh}) \ar@{=}[r] & D^b(\cat{$\left[M^\vee / P_M\right]$-coh}) \ar[r]^-{\text{forget}} & D^b(\cat{$\OO_{M^\vee}$-coh})
} \]

Those induce linear maps (composing further with $\tr$)
\[ \Ext_{\cat{$\overline{M}$-FJ-coh}}^n(\OO, \omega) \rightarrow k \]
which map $\widetilde{p}(\mathcal{E})$ to
\[ p(c_1(\Xi^*\mathcal{E}), \dots, c_n(\Xi^*\mathcal{E}))  \quad \text{ and } \quad  p(c_1(\mathcal{E}), \dots, c_n(\mathcal{E}))  \]
respectively. Here it is used that $\Xi^*$ is an exact functor which is compatible with the tensor product when restricted to locally free (or even torsion-free) objects, that by Theorem~\ref{THEOREMAUTOVECTD} the image
of $J^1(\mathcal{E})'$ under $\Xi^*$ is precisely $J^1(\Xi^*\mathcal{E})$, and that the image under the second functor is $J^1(\mathcal{E})$ where the $P_M$-action on $\mathcal{E}$ is forgotten (by definition of $J^1(\mathcal{E})'$). 

Since there are non-zero Chern polynomials on $M^\vee$, to establish the Theorem, it therefore suffices to show that $\Ext_{\cat{$\overline{M}$-FJ-coh}}^n(\OO, \omega)$ is one dimensional.
This is Proposition~\ref{PROPDIMONE} below. In the compact case, i.e.\@ if $M = \overline{M}$, this is easier and Lemma~\ref{LEMMAOMEGALOG} can be applied directly. 
\end{proof}

\begin{PROP}\label{PROPDIMONE}In the setup of Theorem~\ref{THEOREMHMP}, if $P_M$ is reductive, we have
\[ \dim(\Ext_{\cat{$\overline{M}$-FJ-coh}}^n(\OO, \omega)) = 1. \]
\end{PROP}
\begin{proof}
By Proposition~\ref{PROPRES} we have an exact sequence
\[ \xymatrix{ 0 \ar[r] & \omega   \ar[r] & \omega_{M^\vee} \ar[r] &  \mathcal{D} \ar[r] & 0 } \]
and a finite resolution of the form
\begin{equation}\label{eqres}
 \xymatrix{ 0 \ar[r] & \mathcal{D} \ar[r] & \bigoplus_{Y \text{ codim 1 strata}} \omega_{\overline{Y}} \ar[r] & \bigoplus_{Y \text{ codim 2 strata}} \ar[r] \omega_{\overline{Y}} & \cdots }
 \end{equation}
We get the long exact sequence
\[ \xymatrix{ \Ext^{n-1}(\OO, \mathcal{D} ) \ar[r] & \Ext^{n}(\OO, \omega)   \ar[r] & \Ext^{n}(\OO, \omega_{M^\vee})  \ar[r] &  \Ext^{n}(\OO, \mathcal{D}) } \]
 (all Ext-groups are computed in the category $\cat{$\overline{M}$-FJ-coh}$).
By Lemma~\ref{LEMMAOMEGALOG} below the dimension of $\Ext^{n}(\OO, \omega_{M^\vee})$ is one. Hence it suffices to show that
$\Ext^{n-1}(\OO, \mathcal{D} ) = \Ext^{n}(\OO, \mathcal{D} ) = 0$. Splitting up the exact sequence (\ref{eqres}) into short exact sequences one sees that it suffices to show that 
$\Ext^{i}(\OO, \omega_{\overline{Y}}) = 0$ for $i \ge \dim(Y)$ and for $Y \not= M$. 
We have fully-faithful embeddings (cf.\@ Corollary~\ref{KOREMBEDDINGD})
\[ \xymatrix{ D^b(\cat{$\left[ \cdot/P_Y\right]$-coh}) \ar@{^{(}->}[rr]^-{(0)_{Y,*}} &&  D^b(\cat{$\overline{M}$-FJ-$\ge 0$-coh})) \ar@{^{(}->}[rr]^-{\iota_{0,!}} & & D^b(\cat{$\overline{M}$-FJ-coh}) } \]
such that the image of $\omega_{M^\vee_Y} = (\Lambda^n (\Lie(P_Y)/\Lie(Q_Y)))^*$ under the composition is $\omega_{\overline{Y}}$. 

Furthermore we have
\[ \OO = \iota_{0,!}\,\iota_{0}^*\, \OO.  \]
 
Hence 
\[
\begin{array}{rcll}
 && \Hom_{\mathcal{D}^b(\cat{$\overline{M}$-FJ-coh})}(\iota_{0,!} \,\iota_{0}^*\,\OO, \iota_{0,!}\, (0)_{Y,*}\, \omega_{M^\vee_Y}[i])  \\
 &=& \Hom_{\mathcal{D}^b(\cat{$\overline{M}$-FJ-$\ge 0$-coh})}(\iota_{0}^*\, \OO, (0)_{Y,*}\, \omega_{M^\vee_Y}[i] ) & \text{(fully-faithfulness)} \\
& =& \Hom_{\mathcal{D}^b(\cat{$\left[M^\vee_Y / P_Y\right]$-coh})}(\OO_{M^\vee_Y}, \omega_{M^\vee_Y}[i] ) & \text{(adjunction)}
\end{array}
\]

Therefore the Proposition follows from boundary vanishing condition (axiom B):
\[ H^i(\left[  M^\vee_Y / P_Y \right], \omega_{M^\vee_Y}) = 0 \text{ for } i \ge \dim(Y). \]
\end{proof}

\begin{LEMMA}\label{LEMMAOMEGALOG}
In the setting of Theorem~\ref{THEOREMHMP} we have
\[ \dim(\Ext_{\cat{$\overline{M}$-FJ-coh}}^n(\OO, \omega_{M^\vee})) = 1. \]
\end{LEMMA}
\begin{proof}
We have a fully-faithful embedding (cf.\@ Corollary~\ref{KOREMBEDDINGD})
\[ D^b(\cat{$\left[ M^\vee/P_M \right]$-coh}) \hookrightarrow D^b(\cat{$\overline{M}$-FJ-coh}). \]
The functor $R \Hom(\OO, -)$ is the same as the composition 
\[ D^b(\cat{$\left[ M^\vee/P_M \right]$-coh}) \rightarrow D^b(\cat{$\left[\cdot / P_M\right]$-coh}) \rightarrow D^b(\cat{$\spec(k)$-coh}) \]
where the first functor is the right derived functor of taking global sections and the second is the functor of $P_M$-invariants. However, the last functor is exact (because $P_M$ is reductive) and therefore we have
\[ \Ext_{\cat{$\overline{M}$-FJ-coh}}^n(\OO, \omega_{M^\vee}) = H^n(M^\vee, \omega_{M^\vee})^{P_M}. \]
Since $H^n(M^\vee, \omega_{M^\vee})$ is one-dimensional by Serre duality and thus $P_M$ acts trivially because its center does act trivially on $M^\vee$, the Lemma follows. Note that axiom (T), cf.\@ \ref{AXIOMST}, implies that $n=\dim(\overline{M})=\dim(M^\vee)$. 
\end{proof}
%\begin{BEM}
%For the compact dual associated with a pure Shimura datum the statement of the Lemma can also be seen from a refinement\footnote{With the notation of the proof of Proposition~\ref{PROPBOUNDARYVANISHING}: In this case $U=V=0$. Since $\Lie(R^+)$ and $\Lie(R^-)$ are dual to each other as
%$K$-representations, we get that \[ H^n(Q, \omega)=\left(\left(\Lambda^{n}\Lie(R^+)\right) \otimes \left( \Lambda^{n} \Lie(R^-)\right)\right)^K \] is one dimensional.} of the explicit calculation in the proof of Proposition~\ref{PROPBOUNDARYVANISHING}, which
%shows that the dimension is in fact one. 
%\end{BEM}

\newpage

\bibliographystyle{abbrvnat}
\bibliography{paper5}

\begin{thebibliography}{14}
\providecommand{\natexlab}[1]{#1}
\providecommand{\url}[1]{\texttt{#1}}
\expandafter\ifx\csname urlstyle\endcsname\relax
  \providecommand{\doi}[1]{doi: #1}\else
  \providecommand{\doi}{doi: \begingroup \urlstyle{rm}\Url}\fi

\bibitem[Ash et~al.(2010)Ash, Mumford, Rapoport, and Tai]{AMRT}
A.~Ash, D.~Mumford, M.~Rapoport, and Y.-S. Tai.
\newblock \emph{Smooth compactifications of locally symmetric varieties}.
\newblock Cambridge Mathematical Library. Cambridge University Press,
  Cambridge, second edition, 2010.
\newblock With the collaboration of Peter Scholze.

\bibitem[Atiyah(1957)]{Atiyah}
M.~F. Atiyah.
\newblock Complex analytic connections in fibre bundles.
\newblock \emph{Trans. Amer. Math. Soc.}, 85:\penalty0 181--207, 1957.

\bibitem[Harris(1985)]{Harris85}
M.~Harris.
\newblock Arithmetic vector bundles and automorphic forms on {S}himura
  varieties. {I}.
\newblock \emph{Invent. Math.}, 82\penalty0 (1):\penalty0 151--189, 1985.

\bibitem[Harris(1986)]{Harris86}
M.~Harris.
\newblock Arithmetic vector bundles and automorphic forms on {S}himura
  varieties. {II}.
\newblock \emph{Compositio Math.}, 60\penalty0 (3):\penalty0 323--378, 1986.

\bibitem[Harris and Zucker(1994)]{HZ94}
M.~Harris and S.~Zucker.
\newblock Boundary cohomology of {S}himura varieties. {I}. {C}oherent
  cohomology on toroidal compactifications.
\newblock \emph{Ann. Sci. \'Ecole Norm. Sup. (4)}, 27\penalty0 (3):\penalty0
  249--344, 1994.

\bibitem[Hirzebruch(1958)]{Hir58}
F.~Hirzebruch.
\newblock Automorphe {F}ormen und der {S}atz von {R}iemann-{R}och.
\newblock In \emph{Symposium internacional de topolog\'\i a algebraica
  {I}nternational symposium on algebraic topology}, pages 129--144. Universidad
  Nacional Aut\'onoma de M\'exico and UNESCO, Mexico City, 1958.

\bibitem[H\"ormann(2010)]{Thesis}
F.~H\"ormann.
\newblock \emph{The arithmetic volume of {S}himura varieties of orthogonal
  type}.
\newblock PhD thesis, Humboldt-Universit\"at zu Berlin, 2010.

\bibitem[H{\"o}rmann(2014)]{Hor14}
F.~H{\"o}rmann.
\newblock \emph{The geometric and arithmetic volume of {S}himura varieties of
  orthogonal type}, volume~35 of \emph{CRM Monograph Series}.
\newblock American Mathematical Society, Providence, RI, 2014.

\bibitem[H\"ormann(2016)]{Hor16}
F.~H\"ormann.
\newblock Descent for coherent sheaves along formal/open coverings.
\newblock arXiv: \href{http://arxiv.org/abs/1603.02150}{1603.02150}, 2016.

\bibitem[Jantzen(2003)]{Jantzen}
J.~C. Jantzen.
\newblock \emph{Representations of algebraic groups}, volume 107 of
  \emph{Mathematical Surveys and Monographs}.
\newblock American Mathematical Society, Providence, RI, second edition, 2003.
\newblock ISBN 0-8218-3527-0.

\bibitem[Kashiwara and Schapira(2006)]{KS06}
M.~Kashiwara and P.~Schapira.
\newblock \emph{Categories and sheaves}, volume 332 of \emph{Grundlehren der
  Mathematischen Wissenschaften [Fundamental Principles of Mathematical
  Sciences]}.
\newblock Springer-Verlag, Berlin, 2006.

\bibitem[Mumford(1977)]{Mum77}
D.~Mumford.
\newblock Hirzebruch's proportionality theorem in the noncompact case.
\newblock \emph{Invent. Math.}, 42:\penalty0 239--272, 1977.

\bibitem[Pink(1990)]{Pink}
R.~Pink.
\newblock \emph{Arithmetical compactification of mixed {S}himura varieties}.
\newblock Bonner Mathematische Schriften [Bonn Mathematical Publications], 209.
  Universit\"at Bonn, Mathematisches Institut, Bonn, 1990.
\newblock Dissertation, Rheinische Friedrich-Wilhelms-Universit{\"a}t Bonn,
  Bonn, 1989.

\bibitem[{The Stacks Project Authors}(2014)]{stacks-project}
{The Stacks Project Authors}.
\newblock Stacks project.
\newblock Available at \url{http://stacks.math.columbia.edu}, 2014.

\end{thebibliography}

\end{document}